\documentclass[10pt, a4paper, fleqn]{article}

\usepackage{titling}
\title{Convergence of a Fully Discrete and Energy-Dissipating Finite-Volume Scheme for Aggregation-Diffusion Equations}

\newcommand{\authorPDF}{Bailo, Carrillo, Murakawa, Schmidtchen.}
\newcommand{\subjectPDF}{45K05; 65M08; 65M12.}
\newcommand{\keywordsPDF}{Finite-volume methods; convergence of numerical methods; drift-diffusion equations; integro-differential equations.}
 \usepackage{authblk}

\author[1,2]{Rafael Bailo}
\author[2]{Jos\'{e} A. Carrillo}
\author[3]{\authorcr Hideki Murakawa}
\author[4]{Markus Schmidtchen}

\affil[1]{
Univ. Lille, CNRS, UMR 8524 - Laboratoire Paul Painlev\'{e}
}
\affil[ ]{
	F-59000 Lille, France
}
\affil[ ]{\textit{
		rafael.bailo@univ-lille.fr
	}}

\affil[ ]{}

\affil[2]{
	Mathematical Institute, University of Oxford
}
\affil[ ]{
	OX2 6GG Oxford, United Kingdom
}
\affil[ ]{\textit{bailo@maths.ox.ac.uk, carrillo@maths.ox.ac.uk}}

\affil[ ]{}

\affil[3]{
	Faculty of Advanced Science and Technology, Ryukoku University
}
\affil[ ]{
	1-5 Yokotani, Seta Oe-cho, Otsu, Shiga 520-2194, Japan
}
\affil[ ]{\textit{
		murakawa@math.ryukoku.ac.jp
	}}

\affil[ ]{}

\affil[4]{
Faculty of Mathematics, Technische Universit\"{a}t Dresden
}
\affil[ ]{
	Zellescher Weg 12-14, 01069 Dresden, Germany
}
\affil[ ]{\textit{
		markus.schmidtchen@tu-dresden.de
	}}

\usepackage{geometry}

\usepackage[utf8]{inputenc}
\usepackage[T1]{fontenc}
\usepackage{amsmath}
\usepackage{amsthm}
\usepackage{amsfonts}
\usepackage{amssymb}
\usepackage{graphicx}
\usepackage{mathtools}
\usepackage{ifthen}
\usepackage{dsfont}

\usepackage[UKenglish]{babel}

\usepackage[dvipsnames]{xcolor}

\definecolor{ppGreen}{HTML}{008000}
\definecolor{ppBlue}{HTML}{0000FF}
\definecolor{ppRed}{HTML}{FF0000}
\definecolor{ppPurple}{HTML}{800080}
\definecolor{lightblue}{rgb}{0.145,0.6666,1}
\definecolor{imperialnavy}{RGB}{0,33,71}
\definecolor{imperialblue}{RGB}{0,62,116}
\definecolor{imperialgrey}{RGB}{235,238,238}
\definecolor{imperialcoolgrey}{RGB}{157,157,157}
\definecolor{grey52}{RGB}{52,52,52}
\definecolor{color1}{RGB}{0,62,116}
\definecolor{color2}{RGB}{152,152,152}
\definecolor{color3}{RGB}{52,52,52}
\definecolor{color4}{RGB}{100,100,100} 

\newcounter{review}

\newcommand{\ntcreview}[3]{\refstepcounter{review}

	{\color{#2}{\textbf{[#1]}: #3}}}

\newcommand{\creview}[3]{\ntcreview{#1}{#2}{#3}
	\addcontentsline{tor}{subsection}{\thereview~\textbf{[#1]}:~#3
	}}

\newcommand{\review}[2]{\creview{#1}{blue}{#2}}

\makeatletter

\newcommand\listreviewname{List of Reviews}
\newcommand\listofreviews{\section*{\listreviewname}\@starttoc{tor}}
\makeatother

\usepackage{setspace}
\usepackage[all]{nowidow}

\newcommand{\subjectclassification}[1]{

	{\small\textbf{\textit{AMS Subject Classification --- }} #1}

}
\newcommand{\keywords}[1]{

	{\small\textbf{\textit{Keywords --- }} #1}

}

\renewcommand\lll\MoveEqLeft

\usepackage{tikz}
\usetikzlibrary{patterns}
\usetikzlibrary{decorations.pathreplacing}
\usetikzlibrary{calc}
\usetikzlibrary{fadings}
\usetikzlibrary{shapes}
\usetikzlibrary{plotmarks}
\usetikzlibrary{arrows, decorations.markings}
\tikzset{thicker line small arrows m/.style args={#1in#2}{
			draw=#2,
			solid,
			line width=#1,
			shorten >=1mm,
			decoration={
					markings,
					mark=at position 1.0 with {\arrow[fill=#2,thin]{triangle 90}}
				},
			postaction={decorate}
		}}

\usepackage{pgfplots}
\pgfplotsset{compat=1.15}

\usepackage{float}
\usepackage{subcaption}
\usepackage[section]{placeins}
\usepackage{rotating}

\usepackage{array}
\newcolumntype{L}[1]{>{\raggedright\let\newline\\\arraybackslash\hspace{0pt}}m{#1}}
\newcolumntype{C}[1]{>{\centering\let\newline\\\arraybackslash\hspace{0pt}}m{#1}}
\newcolumntype{R}[1]{>{\raggedleft\let\newline\\\arraybackslash\hspace{0pt}}m{#1}}
\usepackage{booktabs}
\usepackage{tabularx}
\usepackage{multirow}

\usepackage{titling}

\usepackage[
	hypertexnames=false,
	pdftex,
	pdftitle={\thetitle. \authorPDF},
	pdfauthor={\authorPDF},
	pdfsubject={\subjectPDF},
	pdfkeywords={\keywordsPDF},
]{hyperref}
\hypersetup{
	colorlinks=true,
	linkcolor=blue,
	citecolor=blue,
	urlcolor=blue,
	linktocpage
}

\usepackage[capitalise]{cleveref}

 \usepackage{accents}

\newcommand\term\emph

\makeatletter
\let\newtitle\@title
\let\newauthor\@author
\let\newdate\@date
\makeatother
\usepackage{fancyhdr}
\numberwithin{equation}{section}

\usepackage{autobreak}

\usepackage{enumerate}

\newcommand{\singleappendix}{
	\appendix

	\section*{Appendix}

	\addcontentsline{toc}{section}{Appendix}

	\setcounter{equation}{0}
	\renewcommand\theequation{A.\arabic{equation}}

	\setcounter{theorem}{0}
	\renewcommand\thetheorem{A.\arabic{theorem}}
}

\makeatletter
\def\@maketitle{\newpage
	\begin{center}\let \footnote \thanks
		{\LARGE\bfseries \@title \par}\vskip 2.5em{\large
				\lineskip .5em\begin{tabular}[t]{c}\@author
				\end{tabular}\par}\vskip 1em{\large \@date}\end{center}\par
	\vskip 1.5em}
\makeatother



\usepackage{amsthm}
\theoremstyle{plain}
\newtheorem{theorem}{Theorem}[section]
\newtheorem{lemma}[theorem]{Lemma}
\newtheorem{proposition}[theorem]{Proposition}
\newtheorem{corollary}[theorem]{Corollary}

\newtheorem{innercustomthm}{Theorem}
\newenvironment{customthm}[1]
{\renewcommand\theinnercustomthm{#1}\innercustomthm}
{\endinnercustomthm}

\theoremstyle{remark}
\newtheorem{remark}[theorem]{\bf Remark}
\newtheorem{definition}[theorem]{\bf Definition}

\usepackage{esint}

\def\Xint#1{\mathchoice
	{\XXint\displaystyle\textstyle{#1}}{\XXint\textstyle\scriptstyle{#1}}{\XXint\scriptstyle\scriptscriptstyle{#1}}{\XXint\scriptscriptstyle\scriptscriptstyle{#1}}\!\int}
\def\XXint#1#2#3{{\setbox0=\hbox{$#1{#2#3}{\int}$ }
			\vcenter{\hbox{$#2#3$ }}\kern-.6\wd0}}

\def\dashint{\Xint-}

\renewcommand{\th}{\textsuperscript{th} }

\DeclarePairedDelimiter{\prt}{(}{)}
\DeclarePairedDelimiter{\brk}{[}{]}

\DeclarePairedDelimiter{\lbrk}{[}{)}

\DeclarePairedDelimiter{\abs}{|}{|}
\DeclarePairedDelimiter{\norm}{\|}{\|}
\DeclarePairedDelimiter{\set}{\{}{\}}

\DeclarePairedDelimiter{\inn}{\langle}{\rangle}

\usepackage{suffix}

\newcommand{\inner}[2]{\inn{#1,#2}}
\WithSuffix\newcommand\inner*[2]{\inn*{#1,#2}}

\DeclarePairedDelimiter{\floor}{\lfloor}{\rfloor}

\DeclarePairedDelimiter{\positive}{(}{)^{+}}
\DeclarePairedDelimiter{\negative}{(}{)^{-}}

\newcommand\pos\positive
\renewcommand\neg\negative

\WithSuffix\newcommand\pos*{\positive*}
\WithSuffix\newcommand\neg*{\negative*}

\newcommand{\R}{{\mathbb{R}}}

\renewcommand{\L}[1]{{L^{#1}}}
\newcommand{\Lone}{\L{1}}
\newcommand{\Ltwo}{\L{2}}
\newcommand{\Lp}{\L{p}}
\newcommand{\Linf}{\L{\infty}}

\renewcommand{\H}[1]{{H^{#1}}}
\newcommand{\Hone}{\H{1}}

\newcommand{\curlyX}{\mathcal{X}}
\newcommand{\curlyG}{\mathcal{G}}
\newcommand{\OmegaT}{\QT}
\newcommand{\QT}{{Q_T}}

\newcommand{\pnorm}[2]{\norm{#2}_{\L{#1}}}
\WithSuffix\newcommand\pnorm*[2]{\norm*{#2}_{\L{#1}}}

\newcommand{\psnorm}[3]{\norm{#3}_{\L{#1}(#2)}}
\WithSuffix\newcommand\psnorm*[3]{\norm*{#3}_{\L{#1}(#2)}}

\newcommand{\pnormp}[2]{\pnorm{#1}{#2}^{#1}}
\WithSuffix\newcommand\pnormp*[2]{\pnorm*{#1}{#2}^{#1}}

\newcommand{\psnormp}[3]{\psnorm{#1}{#2}{#3}^{#1}}
\WithSuffix\newcommand\psnormp*[3]{\psnorm*{#1}{#2}{#3}^{#1}}

\newcommand\svec\vec

\renewcommand{\vec}{\mathbf}
\renewcommand{\svec}{\boldsymbol}

\newcommand{\conv}{\ast}

\renewcommand{\d}{\mathrm{d}}
\newcommand{\dd}{\mathop{}\!\d}

\newcommand{\pder}[2]{\frac{\partial #1}{\partial #2}}
\newcommand{\secondpder}[2]{\frac{\partial^2 #1}{\partial #2^2}}
\newcommand{\doublepder}[3]{\frac{\partial^2 #1}{\partial #2 \partial #3}}

\newcommand{\dt}{\dd t}
\newcommand{\ds}{\dd s}
\newcommand{\dx}{\dd x}

\newcommand{\px}{\partial_x}

\newcommand{\Dt}{\Delta t}
\newcommand{\Dx}{\Delta x}

\newcommand{\ddx}{\d_{x}}

\newcommand{\h}{_{h}}

\newcommand{\nhalf}{1/2}

\renewcommand{\i}{_{i}}
\newcommand{\ip}{_{i+1}}

\newcommand{\ih}{_{i+\nhalf}}
\newcommand{\imh}{_{i-\nhalf}}

\renewcommand{\j}{_{j}}

\newcommand{\jm}{_{j-1}}

\newcommand{\jmh}{_{j-\nhalf}}

\renewcommand{\k}{_{k}}
\newcommand{\kp}{_{k+1}}

\newcommand{\kh}{_{k+\nhalf}}

\renewcommand{\l}{_{l}}

\newcommand{\n}{^{n}}
\newcommand{\np}{^{n+1}}

\newcommand{\nss}{^{**}}

\newcommand{\ppr}{(r)}

\newcommand{\Wr}{^{W,\,\ppr}}

\newcommand{\Wik}{W_{i-k}}

\newlength{\dhatheight}

\ifthenelse{\isundefined{\Wr}}{
	\newcommand{\Wr}{^{W,\,\ppr}}
}{
	\renewcommand{\Wr}{^{W,\,\ppr}}
}

\DeclareMathOperator{\argmax}{argmax}
\DeclareMathOperator{\argmin}{argmin}

\newcommand{\curlyN}{\mathcal{N}}
\newcommand{\curlyI}{\mathcal{I}}
\newcommand{\curlyC}{\mathcal{C}}

\newcommand{\cH}{\curlyC_{H'}}
\newcommand{\cinf}{\curlyC_{\infty}}
\newcommand{\cK}{\curlyC_{K}}
\newcommand{\cdH}{\curlyC_{\partial_x H'}}
\newcommand{\cdR}{\curlyC_{\partial_x \rho}}
\newcommand{\cVone}{\curlyC_V^{\prt{1}}}
\newcommand{\cVtwo}{\curlyC_V^{\prt{2}}}

\newcommand{\Dl}{\Delta l}
\newcommand{\Dh}{\Delta h}

\newcommand{\K}{^{n+K}}
\newcommand{\Kp}{^{n+K+1}}

\newcommand{\nk}{^{n+k}}

\newcommand{\nkm}{^{n+k-1}}

\newcommand{\ik}{_{i+k}}

\newcommand{\ikm}{_{i+k-1}}

\newcommand{\ikmh}{_{i+k-1/2}}

\newcommand{\iK}{_{i+K}}
\newcommand{\iKp}{_{i+K+1}}

\newcommand{\Emean}[2]{\mu_{H}\prt{#1,#2}}

\usepackage{algorithm,algpseudocode} \usepackage{todonotes}

\allowdisplaybreaks

\renewcommand{\review}[2]{}
\renewcommand{\creview}[3]{}
\renewcommand{\ntcreview}[3]{}

\renewcommand{\tableofcontents}{}
\renewcommand{\listofreviews}{}

\expandafter\def\csname ver@etex.sty\endcsname{3000/12/31}

\usepackage{autonum}
\makeatletter
\patchcmd{\autonum@saveEnvironmentSubcommands}
{(0,0)\begin}
{(0,0)\hfuzz=\maxdimen\begin}
{}{}
\makeatother

\AtBeginDocument{

}

\makeatletter
\autonum@generatePatchedReferenceCSL{ref}
\autonum@generatePatchedReferenceCSL{eqref}
\autonum@generatePatchedReferenceCSL{Cref}
\makeatother

\newcommand{\revision}[2]{#2}
\newcommand{\revisionNote}[2]{}

\definecolor{revisionColourOne}{RGB}{180,0,0}
\definecolor{revisionColourTwo}{RGB}{0,0,180}

\newcommand{\revisionOne}[1]{\revision{revisionColourOne}{#1}}
\newcommand{\revisionTwo}[1]{\revision{revisionColourTwo}{#1}}

\newcommand{\revisionNoteOne}[1]{\revisionNote{revisionColourOne}{#1}}
\newcommand{\revisionNoteTwo}[1]{\revisionNote{revisionColourTwo}{#1}}

\begin{document}
\begin{singlespace}\maketitle\end{singlespace}

\begin{abstract}
	We study an implicit finite-volume scheme for non-linear, non-local aggregation-diffusion equations which exhibit a gradient-flow structure, recently introduced in \cite{B.C.H2020}. Crucially, this scheme keeps the dissipation property of an associated fully discrete energy, and does so unconditionally with respect to the time step. Our main contribution in this work is to show the convergence of the method under suitable assumptions on the diffusion functions and potentials involved.
\end{abstract}

\subjectclassification{\subjectPDF}
 
\revisionNoteOne{Changes corresponding to the comments of Reviewer 1 are shown in red.}
\revisionNoteTwo{Similarly, changes corresponding to Reviewer 2 are shown in blue.}

\tableofcontents
\listofreviews

\section{Introduction}

This work is concerned with the analysis of a finite-volume scheme for the family of \term{aggregation-diffusion equations}
\begin{subequations} \label{eq:continuous}
\begin{equation} \label{eq:continuous_eq}
	\pder{\rho}{t} = \pder{}{x}\brk*{\rho\pder{}{x}\prt{H'(\rho)+V(x)+W(x)\conv\rho}}
\end{equation}
on $\Omega = \prt{-L,L}$, for some $L>0$, and for $t>0$, together with a given non-negative initial datum
\begin{align+}
\rho(0,x) = \rho_0(x).
\end{align+}
The equation is equipped with no-flux boundary conditions:
\begin{align+} \label{eq:continuous_no_flux}
\rho\pder{}{x}\prt{H'(\rho)+V(x)+W(x)\conv\rho} = 0
\end{align+}
\end{subequations}
at $x=\pm L$.
The unknown, $\rho$, typically models the \term{density} of particles of a given species.
The term $H$ is called the \term{internal energy density}, and models linear or non-linear diffusion effects within the species $\rho$.
The potential $V$ models the \term{drift} of the species; it is sometimes referred to as the \term{external potential}, since it provides an energetic landscape that drives the density towards favourable regions. Finally, the potential $W$ models the interactions of particles with each other; we call it the \term{internal potential} or the \term{interaction potential}.

It is important to highlight the generality of \cref{eq:continuous}, since it encompasses many popular equations. In the absence of the potential terms, it becomes the well-known \term{filtration equation}. Linear diffusion can be obtained by choosing Boltzmann's entropy as the internal energy, i.e. $H(\rho) = \rho \log\prt{\rho} - \rho$; moreover, the linear Fokker-Planck equation is recovered by letting $V(x)=x^2/2$. The \term{porous medium equation} \cite{Vazquez2006} can be found instead by considering a power law for the internal energy density: $H(\rho)=\rho^m/\prt{m-1}$, for some $m>1$. Typically, the non-linearity accounts for a density-dependent diffusion, as observed, for example, in porous medium flow \cite{Barenblatt1952}, in heat transfer \cite{Z.K1950}, or in population dynamics \cite{G.N1975}. The interplay between diffusion and potential effects has a myriad of applications in life sciences; see \cite{C.C.Y2019} for a recent survey.

The choice of the potential $W$ models certain phenomena, such as attraction, repulsion, and combinations thereof. Pairwise interactions often depend solely on the distance between agents; as a result, many applications assume the interaction potential to be radially symmetric, i.e., $W(x) = w(|x|)$, for some $w: \R \rightarrow \R $. This radial potential is called \term{attractive} if $w'>0$ and \term{repulsive} if $w'<0$. Typical choices include power laws in the context of granular media \cite{B.C.P1997, Toscani2010}, or combinations of power laws ($W(x)=|x|^a/a - |x|^b/b$), which feature short-range repulsion and long-range attraction, in swarming and population dynamics \cite{L.L.E2008,F.H.K2011,K.S.U+2011,B.K.S+2015,C.H.S2018}; these effects may appear simple at first glance, but they result in complex and subtle dynamics \cite{C.D.M2016,B.C.L+2013,F.H2013,K.S.U+2011}. Other choices include local sensitivity regions (modelled by indicator functions \cite{M.E1999,T.B2004,T.B.L2006} or attractive-repulsive Morse potentials \cite{D.C.B+2006,L.L.E2008,Fetecau2011,C.M.P2013,C.H.M2014}), as well as growth and saturation terms \cite{A.P.S2006,M.T2015,V.S2015,C.M.S+2019}.

The equation, in the general form \eqref{eq:continuous}, has been studied in \cite{C.C.H2015,C.M.V2003,C.M.V2006,McCann1997} through its dissipative structure. To be precise, the energy functional
\begin{align} \label{eq:continuousEnergy}
	E(\rho) \coloneqq \int \prt*{H(\rho)+V\rho +\frac{1}{2}(W\conv \rho )\rho}\dx
\end{align}
is dissipated along solutions of \cref{eq:continuous}; formally, we may compute
\begin{align} \label{eq:energydissipation}
	\revisionTwo{
		\frac{\d}{\dt}E(\rho) = - \int \rho \abs*{\pder{}{x}\prt{H'(\rho) + V + W\conv \rho}}^2 \dx.
	}
\end{align}

Numerical schemes which preserve the energy-dissipation structure are desirable to simulate the behaviour of \cref{eq:continuous}. Suitable methods should also \revisionOne{ensure} the conservation of mass and preserve the non-negativity of the solutions. Such schemes were proposed in \cite{C.F2007}, for non-linear diffusion equations with drift, and in \cite{C.C.H2015}, for equations including non-local interactions. In the latter, the authors propose a semi-discrete, upwind finite-volume scheme that preserves the entropic structure of the equation at the semi-discrete level. Moreover, by construction, the method conserves the total mass and preserves the non-negativity of solutions also. A fully discrete scheme with the same properties was proposed in \cite{B.C.H2020}.

The generalisation of \cref{eq:continuous} to the case of two interacting species was studied in \cite{C.F.S2020,C.H.S2018} by extending the aforementioned semi-discrete method. Although the system generally lacks an entropic structure (even at the continuous level), the discretisation resembles strongly that of \cite{C.C.H2015}.
A fully discrete, implicit finite-volume scheme for the same system was additionally developed by \cite{C.F.S2020}. The authors prove the convergence of their methods (semi-discrete and fully discrete) to weak solutions of the equations through a technique of flow interchange: they construct an auxiliary functional whose dissipation contains important gradient information. The construction for a general non-linearity $H$ had already been employed in \cite{C.F.S+2018} in a semi-discrete setting, and it was adapted to prove convergence of the numerical method of \cite{C.F.S2020}.

We present a sketch of the estimate, adapted to \cref{eq:continuous}, at the continuous level. Given a convex internal energy density, $H$, we construct a function $K$ via \revisionTwo{$H''(s) = sK''(s)$}, for $s>0$. We readily find:
\begin{align} \label{eq:flowinterchange_continuous_formal}
	\begin{split}
		\frac{\d}{\dt} \int K(\rho) \dx
		&= \int K'(\rho) \pder{\rho}{t}\dx\\
		&= \int K'(\rho) \pder{}{x} \brk*{\rho\pder{}{x}\prt{H'(\rho)+V(x)+W(x)\conv\rho}}\dx\\
		&= - \int \rho \pder{K'}{x}(\rho) \pder{}{x}\prt{H'(\rho)+V(x)+W(x)\conv\rho}\dx\\
		&= - \int \abs*{\pder{}{x}H'(\rho)}^2 \dx - \int \pder{}{x} H'(\rho) \pder{}{x}\prt*{V(x) + W\conv\rho}\dx\\
		&\leq -(1-\alpha)\int \abs*{\pder{}{x}H'(\rho)}^2 \dx + \frac{1}{2\alpha}\int \abs{V'(x) + W'\conv\rho}^2\dx\\
		&\leq C  -(1-\alpha)\int \abs*{\pder{}{x}H'(\rho)}^2 \dx,
	\end{split}
\end{align}
\revisionTwo{having used integration by parts together with the no-flux boundary conditions,} and where $\alpha\in (0,1)$ arises from Young's inequality.

To argue the convergence of the method of \cite{B.C.H2020}, we will construct a family of appropriate interpolations of the discrete numerical approximations, which will be identified as elements of an $\Linf$ class. The adaptation of the flow interchange argument \eqref{eq:flowinterchange_continuous_formal}, accompanied by some additional \term{a priori} estimates, will show the strong convergence of the approximating sequence to a weak solution of \cref{eq:continuous}. As usual, these arguments do not yield a specific rate of convergence, though thorough numerical experiments in \cite{B.C.H2020} suggest first-order accuracy. Nevertheless, we establish a convergence result for their numerical method, filling this theoretical gap in the literature.

Future work might address the question of passing to the limit in the energy-dissipation inequality. While it is clear that $E(\rho(t)) \leq E(\rho(s))$ holds whenever $t\geq s$, the analysis of the limit of the non-linear Fisher information remains open. Furthermore, it would be interesting to ascertain if the long-time behaviour of \cref{eq:continuous} under suitable convex potentials is preserved by the fully-discrete scheme. This behaviour is well-understood at the continuous level \cite{C.M.V2003,A.G.S2005}, but, to the best of our knowledge, discrete results are only known for particular cases \cite{B.C.F2015,B.C2017}.

The rest of this work is organised as follows: in \cref{sec:numericalscheme} we recapitulate the numerical method, introduce the notion of weak solution, and present the main result;  \cref{sec:apriori_estimates} is dedicated to establishing the existence of solutions to the numerical method as well as proving certain \term{a priori} bounds; in \cref{sec:compactness} we employ the estimates to deduce strong compactness of the approximating sequence, and weak compactness of the discrete velocity term (which comprises the internal energy, the drift, and interaction terms); \cref{sec:convergence_scheme} is dedicated to proving the main result: the convergence of the approximating sequence to a weak solution of \Cref{eq:continuous}; finally, in \Cref{sec:linear} we adapt our convergence result to the case of linear diffusion; \revisionOne{the Boltzmann entropy is not covered by the main result because the upwind discretisation of the scheme is incompatible with the infinite speed of propagation associated with linear diffusion; nevertheless, the argument can be modified to prove the convergence of the scheme.}
 \section{Numerical Scheme} \label{sec:numericalscheme}

In this section we introduce the fully-discrete, implicit method of \cite{B.C.H2020} for the problem \eqref{eq:continuous} in one dimension, i.e.,
\begin{equation} \tag{\ref{eq:continuous_eq} revisited}
	\pder{\rho}{t} = \pder{}{x}\brk*{\rho\pder{}{x}\prt{H'(\rho)+V(x)+W(x)\conv\rho}}.
\end{equation}
The equation is posed on the domain $\Omega\coloneqq\prt{-L,L} \subset \R$, and a time interval $\prt{0,T}$. The equation is supplemented with no-flux boundary  conditions and a given non-negative initial datum, $\rho\prt{0,x}=\rho_0(x) \in L^\infty(\Omega)$. Throughout this work, we will denote the space-time cylinder by $\QT=\prt{0,T}\times\Omega$.

First, we introduce the discretisation of the domain.
\begin{definition} [Discretisation of the domain $\QT$]
	Let $M, N\in \mathbb{N}$ be two integers. The spatial domain, $\Omega$, is divided into $2M$ uniform cells of length $\Dx=L/M$. The \term{$i$\th cell}, denoted by $C\i\coloneqq\lbrk{x\imh,x\ih}$, is centred at $x\i$; the \term{cell centre} is located at  $x\i=-L+\Dx\prt{i-1/2}$, as shown in \cref{fig:discretisation}.

	The time domain is divided into $N+1$ equal intervals of length $\Dt=T/(N+1)$, for some integer $N\in \mathbb{N}$. The $n$\th interval is defined by $I\n\coloneqq\lbrk{t\n,t\np}$, with $t\n\coloneqq n\Dt$.

	These partitions give rise to a \term{computational mesh}, $\set{Q\i\n}$ which divides the space-time cylinder, $\QT$, into a family of \emph{finite-volume cells}, denoted by $Q\i\n=I\n \times C\i$, for $\revisionTwo{n\in \curlyN\coloneqq\set{0,\ldots, N}}$ and $i\in\curlyI\coloneqq\set{1, \ldots, 2M}$. The coarseness of a given mesh is measured by the \emph{mesh size} $h=\Dx=c\Dt$, for a fixed $c>0$.

	Finally, we define the \emph{dual cells}, $C_{i+1/2}\coloneqq\lbrk{x\i,x\ip}$, for $i \in \set{1,\dots, 2M-1}$, which make up the \emph{dual mesh}, $\set{Q\ih\n}$, where $Q\ih\n=I\n \times C\ih$.
\end{definition}

\begin{figure}[ht]
	\centering
	\tikzset{
	hatch distance/.store in=\hatchdistance,
	hatch distance=10pt,
	hatch thickness/.store in=\hatchthickness,
	hatch thickness=2pt
}

\makeatletter
\pgfdeclarepatternformonly[\hatchdistance,\hatchthickness]{flexible hatch}
{\pgfqpoint{0pt}{0pt}}
{\pgfqpoint{\hatchdistance}{\hatchdistance}}
{\pgfpoint{\hatchdistance-1pt}{\hatchdistance-1pt}}{
	\pgfsetcolor{\tikz@pattern@color}
	\pgfsetlinewidth{\hatchthickness}
	\pgfpathmoveto{\pgfqpoint{0pt}{0pt}}
	\pgfpathlineto{\pgfqpoint{\hatchdistance}{\hatchdistance}}
	\pgfusepath{stroke}
}
\makeatother

\begin{tikzpicture}[scale=0.9]
	\def\l{1.6}
	\def\h{0.2}
	\def\dh{2}

	\draw[thick] ({\l*2},{-\h}) -- ({\l*2},{\h});
	\draw[thick] ({\l*2.5},{0}) -- ({\l*2.5},{0.5*\h}) node[above=-2pt] {$x_1$};
	\draw[thick] ({\l*3},{-\h}) -- ({\l*3},{\h});
	\draw[thick] ({\l*3.5},{0}) -- ({\l*3.5},{0.5*\h}) node[above=-2pt] {$x_2$};
	\draw[thick] ({\l*4},{-\h}) -- ({\l*4},{\h});

	\draw[thick] ({\l*5.5},{-\h}) -- ({\l*5.5},{\h});
	\draw[thick] ({\l*6.5},{-\h}) -- ({\l*6.5},{\h});

	\draw[thick] ({\l*8},{-\h}) -- ({\l*8},{\h});
	\draw[thick] ({\l*8.5},{0}) -- ({\l*8.5},{0.5*\h}) node[above=-2pt] {$x_{2M-1}$};
	\draw[thick] ({\l*9},{-\h}) -- ({\l*9},{\h});
	\draw[thick] ({\l*9.5},{0}) -- ({\l*9.5},{0.5*\h}) node[above=-2pt] {$x_{2M}$};
	\draw[thick] ({\l*10},{-\h}) -- ({\l*10},{\h});

	\draw[thick] ({\l*2},0) --     ({\l*4},0);
	\draw[thick] ({\l*5.5},0) --     ({\l*6.5},0);
	\draw[thick,dotted] ({\l*4},0) -- ({\l*8},0);
	\draw[thick] ({\l*8},0) --     ({\l*10},0);

	\node[below=5pt] at ({\l*(2+0.5)},0) {$C_{1}$};
	\node[below=5pt] at ({\l*(3+0.5)},0) {$C_{2}$};
	\node[below=5pt] at ({\l*(8+0.5)},0) {$C_{2M-1}$};
	\node[below=5pt] at ({\l*(9+0.5)},0) {$C_{2M}$};
	\node[below=5pt] at ({\l*(6)},0) {$C\i$};
	\draw[thick] ({\l*6},{0}) -- ({\l*6},{0.5*\h}) node[above=-2pt] {$x\i$};

	\draw[<->,thick] ({\l*3},-0.8) -- ({\l*4},-0.8) node[midway, below] {$\Dx$};

	\draw[->] ({\l*(2-0.3)},{-\h*(1+3)}) node[below]{$x_{1/2}=-L$} -- ({\l*(2-0.05)},{-\h*(1.5)});
	\draw[->] ({\l*(10+0.3)},{-\h*(1+3)}) node[below]{$x_{2M+1/2}=L$} -- ({\l*(10+0.05)},{-\h*(1.5)});

	\draw[->] ({\l*(5.5-0.2)},{-\h*(1+3)}) node[below]{$x\imh$} -- ({\l*(5.5-0.05)},{-\h*(1.5)});
	\draw[->] ({\l*(6.5+0.2)},{-\h*(1+3)}) node[below]{$x\ih$} -- ({\l*(6.5+0.05)},{-\h*(1.5)});

	\draw[white, thin, pattern=flexible hatch, hatch distance=5pt, hatch thickness=0.25pt, pattern color=black!50!white]
	({\l*2.0},{-\h+\dh}) rectangle
	({\l*2.5},{+\h+\dh});

	\draw[white, thin, pattern=flexible hatch, hatch distance=5pt, hatch thickness=0.25pt, pattern color=black!50!white]
	({\l*9.5},{-\h+\dh}) rectangle
	({\l*10.0},{+\h+\dh});

	\draw[thick] ({\l*2.5},{-\h+\dh}) -- ({\l*2.5},{\h+\dh});
	\draw[thick] ({\l*3},{0+\dh}) -- ({\l*3},{0.5*\h+\dh}) node[above=-2pt] {$x_{1+\nhalf}$};
	\draw[thick] ({\l*3.5},{-\h+\dh}) -- ({\l*3.5},{\h+\dh});

	\draw[thick] ({\l*5},{-\h+\dh}) -- ({\l*5},{\h+\dh});
	\draw[thick] ({\l*6},{-\h+\dh}) -- ({\l*6},{\h+\dh});
	\draw[thick] ({\l*7},{-\h+\dh}) -- ({\l*7},{\h+\dh});

	\draw[thick] ({\l*8.5},{-\h+\dh}) -- ({\l*8.5},{\h+\dh});
	\draw[thick] ({\l*9},{0+\dh}) -- ({\l*9},{0.5*\h+\dh}) node[above=-2pt] {$x_{2M-\nhalf}$};
	\draw[thick] ({\l*9.5},{-\h+\dh}) -- ({\l*9.5},{\h+\dh});

	\draw[thick] ({\l*2.5},{0+\dh}) --     ({\l*3.5},{0+\dh});
	\draw[thick] ({\l*5.0},{0+\dh}) --     ({\l*7.0},{0+\dh});
	\draw[thick,dotted] ({\l*3.5},{0+\dh}) -- ({\l*8.5},{0+\dh});
	\draw[thick] ({\l*8.5},{0+\dh}) --     ({\l*9.5},{0+\dh});

	\node[below=5pt] at ({\l*(3+0.0)},{0+\dh}) {$C_{1+\nhalf}$};
	\node[below=5pt] at ({\l*(8+1.0)},{0+\dh}) {$C_{2M-\nhalf}$};
	\node[below=5pt] at ({\l*(5.5)},{0+\dh}) {$C\imh$};
	\node[below=5pt] at ({\l*(6.5)},{0+\dh}) {$C\ih$};
	\draw[thick] ({\l*5.5},{0+\dh}) -- ({\l*5.5},{0.5*\h+\dh}) node[above=-2pt] {$x\imh$};
	\draw[thick] ({\l*6.5},{0+\dh}) -- ({\l*6.5},{0.5*\h+\dh}) node[above=-2pt] {$x\ih$};
\end{tikzpicture} 	\caption{Discretisation of the spatial domain $\Omega$.}
	\label{fig:discretisation}
\end{figure}

We now proceed to discretise the problem \eqref{eq:continuous}. First, we construct the discrete initial datum by \revisionOne{$\rho^0\coloneqq\set{\rho\i^0}_{i\in\curlyI}$}, on the discretised spatial domain through the cell averages of the continuous datum:
\begin{align} \label{eq:cellaveragedata}
	\rho\i^0 = \dashint_{C\i} \rho_0(x) \dx,
\end{align}
where $\dashint_{C\i} f(s) \ds = \frac{1}{\abs{C\i}} \int_{C\i} f(s) \ds$ denotes the average of $f$ on the $i$\th cell $C\i$. To discretise \Cref{eq:continuous_eq}, we integrate it over a test cell, $Q\i\n$, which yields
\begin{align}
	\int_{C\i} \rho(t\np, x)\dx - \int_{C\i} \rho(t\n, x)\dx + \int_{I\n} F(t, x\ih) - F(t, x\imh)\dt = 0,
\end{align}
where
\begin{align}
	\revisionTwo{F\prt{t,x} \coloneqq -\rho\pder{}{x}\prt{H'(\rho)+V(x)+W(x)\conv\rho} }
\end{align}
is the \term{flux}. This identity can be approximated by
\begin{align}
	\Dx \prt*{\rho\np\i - \rho\n\i} + \Dt \prt*{F\np\ih - F\np\imh}= 0.
\end{align}
Giving rise to the numerical scheme
\begin{subequations} \label{eq:discrete}
\begin{align+}
\frac{\rho\i\np-\rho\i\n}{\Dt} = -\frac{F\ih\np-F\imh\np}{\Dx},
\end{align+}
for $i\in\curlyI$, and $n\in\curlyN$. The \term{discrete solution}, $\rho\i\n$, approximates the continuous solution in the finite-volume sense,
\begin{align} \label{eq:cellaverage}
	\rho\i\n \simeq \dashint_{C\i} \rho(t\n,x) \dx.
\end{align}
The \term{numerical fluxes}, $F\ih\n$, are given by the upwind discretisation
\begin{align+}
F\ih\np = \rho\i\np\pos{u\ih\np} + \rho\ip\np\neg{u\ih\np},
\end{align+}
where $\pos{s}\coloneqq\max\set{s,0}$ and $\neg{s}\coloneqq\min\set{s,0}$, for any $s\in \R$, denote respectively the positive and negative parts of $s$. Choosing $F_{\nhalf}\np = F_{2M+\nhalf}\np = 0$ incorporates the no-flux boundary conditions into the scheme.

In order to define the \term{velocities}, $u\ih\np$, it is convenient to introduce the \term{discrete entropy variables},
\begin{align+} \label{eq:discrete_entropy_variables}
\xi\i\np = H'(\rho\i\np ) + V\i + (W\conv\rho\nss)\i,
\end{align+}
and to set
\begin{align+}
u\ih\np = -\frac{\xi\ip\np -\xi\i\np }{\Dx}.
\end{align+}
The contributions from the external and interaction potentials are discretised as
\begin{align+}
V\i = \dashint_{C\i} V(s) \ds,
\end{align+}
and $(W\conv\rho\nss)\i=\sum_{j=1}^{2M}W_{i-j}\rho_j\nss\Dx$, where
\begin{align+}
W_{i-j} = \dashint_{C\j} W(x\i-s) \ds.
\end{align+}
\end{subequations}

At this stage, we highlight once more that the evolution of the density, $\rho$, is governed by an entropic part, which works to minimising an internal energy, as well as a potential part, that generates a drift. For convenience of the analysis, we will hereafter split these contributions at the discrete level:
\begin{subequations} \label{eq:velocity_split_entropic_potential}
\begin{align+}
u\ih\np = h\ih\np + v\ih\np,
\end{align+}
with
\begin{align+}
h\ih\np \coloneqq -\frac{H'(\rho\ip\np) - H'(\rho\i\np)}{\Dx}
\end{align+}
and
\begin{align+}
v\ih\np \coloneqq -\frac{V\ip - V\i}{\Dx} - \frac{(W\conv\rho\nss)\ip - (W\conv\rho\nss)\i}{\Dx}.
\end{align+}
\end{subequations}

\begin{remark}[Choice of $\rho\nss$]
	The density present in the interaction term, $\rho\nss$, may be chosen as
	$$
		\rho\i\nss \in \set*{\rho\i\n, \rho\i\np,\frac{\rho\i\n+\rho\i\np}{2}},
	$$
	with $i\in\curlyI$ and $n\in\curlyN$. As discussed in \cite[Theorem  3.9]{B.C.H2020}, this choice depends on certain properties of the potential and serves to establish a discrete analogue of the energy dissipation \eqref{eq:energydissipation}. In fact, it can be shown that the discrete energy is unconditionally dissipated if
	\begin{enumerate}[(i)]
		\item $\rho^{**}=\prt{\rho^n+\rho^{n+1}}/{2}$;
		\item $\rho^{**}=\rho^n$ and the potential $W$ is negative definite;
		\item $\rho^{**}=\rho^{n+1}$ and the potential $W$ is positive definite.
	\end{enumerate}
	The interaction potential, $W$, is called negative definite (resp. positive definite) if
	\begin{align}
		\sum_{i,j=1}^{2M}\Wik(\eta\i-\zeta\i)(\eta\j-\zeta\j)\leq 0 \qquad \prt{\textrm{resp. } \geq 0},
	\end{align}
	for any two vectors $\eta$ and $\zeta$.
	For the readers' convenience, we reiterate how the discrete energy-dissipation inequality is obtained in \cref{sec:existence}.

	In the sequel, we will only consider the most general case: $\rho^{**} = (\rho\n + \rho\np)/2$. Nonetheless, the analysis is the same for the other choices of $\rho^{**}$, since the quantity only plays a small role in the fixed-point argument of Theorem \ref{thm:existence} and appears inside the (conserved) $\Lone$-norm elsewhere.
\end{remark}

Before proving the analytical properties of scheme \eqref{eq:discrete}, we introduce the necessary notation for the subsequent sections.

\begin{definition} [Piecewise constant interpolation]
	Given a discrete solution $\set{\rho\i\n}$, we define the piecewise constant interpolations
	\begin{align} \label{eq:interpolation}
		\rho\h(t,x) \coloneqq \rho\i\n \textrm{ for } \prt{t,x}\in Q\i\n,
	\end{align}
	for $i\in\curlyI$ and $n\in\curlyN$.  Moreover, for any function $\eta:\R \to \R$, we use the notation $\eta\i\n\coloneqq\eta(\rho\i\n)$, as well as $\eta\h \coloneqq \eta \circ \rho\h$.

	Given a quantity defined at the cell interfaces, $x\ih$, for $i=1,\dots,2M-1$ (such as the velocity $u\ih\n$ and the flux $F\ih\n$), we define its associated piecewise constant interpolations as
	\begin{align} \label{eq:interpolationdual}
		\zeta\h(t,x) \coloneqq \zeta\ih\n \textrm{ for } \prt{t,x}\in Q\ih\n,
	\end{align}
	and $\zeta\h(t,x) = 0$ whenever  if $x<x_{1/2}$ or $x>x_{2M+1/2}$.
\end{definition}

We now establish a bridge between the discrete approximations from scheme \eqref{eq:discrete} and the $\Lp$-functions through the piecewise constant interpolations, using the standard norms:
\begin{align}
	\psnorm{p}{\Omega}{\eta\h(t)}
	=\prt*{\int_{\Omega}\abs{\eta\h(t,x)}^p\dx}^{1/p}
	=\prt*{\sum_{i=1}^{2M}\abs{\eta\i\n}^p\Dx}^{1/p},
\end{align}
for $t\in I\n$, as well as
\begin{align}
	\psnorm{p}{\QT}{\eta\h}
	=\prt*{\int_{0}^{T}\int_{\Omega}\abs{\eta\h(t,x)}^p\dx\dt}^{1/p}
	=\prt*{\sum_{n=0}^{N}\sum_{i=1}^{2M}\abs{\eta\i\n}^p\Dx\Dt}^{1/p};
\end{align}
these are analogously described for quantities defined on the dual mesh.

Last, but not least, we define the discrete gradients, which occur naturally at the cell interfaces in our analysis.
\begin{definition}[Discrete gradients]
	We define the discrete gradient, $\ddx\eta\h$, for a function, $\eta\h$, as
	\begin{align}
		\ddx\eta\n\ih\coloneqq\frac{\eta\n\ip-\eta\n\i}{\Dx} \quad\textrm{and}\quad \ddx\eta\h\prt{t, x} = \ddx\eta\n\ih,
	\end{align}
	for $\prt{t,x}\in Q\ih\n$ and $1\leq i\leq 2M-1$.

	In the same vein, the discrete gradient of a quantity defined on the dual mesh is given by
	\begin{align}
		\ddx\zeta\i\n\coloneqq\frac{\zeta\ih\n-\zeta\imh\n}{\Dx} \quad\textrm{and}\quad \ddx\zeta\h\prt{t, x} = \ddx\zeta\n\i,
	\end{align}
	for $\prt{t,x}\in Q\i\n$ and $i\in\curlyI$. Note that the values of $\zeta\n_{1/2}$ and $\zeta\n_{2M+1/2}$ should either be clear from the boundary conditions or otherwise prescribed.
\end{definition}

We conclude this section by introducing our notion of weak solution and stating the main convergence result.

\begin{definition} [Weak solutions] \label{def:weaksolution}
	A function $\rho\in \Linf(\QT)$ is a weak solution to \Cref{eq:continuous} if it satisfies
	\begin{align}
		\revisionTwo{\int_0^T\int_\Omega \rho \pder{\varphi}{t} - \rho \pder{\varphi}{x} \pder{}{x} \prt*{H'(\rho)  + V + W\conv\rho} \dx \dt = -\int_\Omega \rho(0) \varphi(0) \dx,}
	\end{align}
	for any smooth test function $\varphi\in C^\infty(\overline{Q_T})$ such that $\varphi(T)=0$.
\end{definition}

\begin{theorem} [Convergence of the scheme] \label{th:convergence}
	Suppose that $V, W \in C^2([-L, L])$, and also that $H\in C^2([0,\infty))$, such that $H''(s)>0$ for all positive values of $s$, as well as \revisionTwo{$s^{-1} H''(s) \in L^1_{\mathrm{loc}}([0,\infty))$}. Moreover, suppose that there exists a function, $K$, bounded from below, such that $sK''(s) = H''(s)$.
	Then, for any non-negative initial datum $\rho_0\in\Linf(\Omega)$, we find that:
	\begin{enumerate}[(i)]
		\item scheme \eqref{eq:discrete} admits a non-negative solution that preserves the initial mass regardless of the mesh size;
		\item under the condition ${\curlyC}_V^{(2)}\Dt < 1$ (viz. \cref{th:aprioriLinfboundrho}), the associated piecewise constant interpolation, $\rho_h$, converges strongly in any $\Lp(\QT)$, for $1\leq p<\infty$, up to a subsequence;
		      \revisionOne{here, $\cVtwo = \psnorm{\infty}{\Omega}{V''} + \psnorm{\infty}{\Omega}{W''} \psnorm{1}{\Omega}{\rho_0}$;}
		\item the limit is a weak solution to \Cref{eq:continuous} in the sense of \Cref{def:weaksolution}.
	\end{enumerate}
\end{theorem}

\begin{remark}[Assumptions throughout this work]
	\revisionTwo{
		In the sequel, we assume that the assumptions of \cref{th:convergence} are satisfied, unless otherwise stated.
	}
\end{remark}

\begin{remark}[Assumptions on $H$, $V$, and $W$]
	\revisionTwo{
	\cref{th:convergence} makes the assumption $s^{-1} H''(s) \in L^1_{\mathrm{loc}}([0,\infty))$. This property is employed in the proof of \cref{th:discreteHgradientestimate}, in order to ensure that the function $K'$ is defined at zero. In terms of common choices of $H$, the assumption rules out linear diffusion, corresponding to $H(\rho) = \rho \log\prt{\rho} - \rho$, and porous medium diffusion, $H(\rho)=\rho^m/\prt{m-1}$, whenever $m\leq 2$.
	}

	\revisionTwo{
		The result of \cref{th:discreteHgradientestimate} can nevertheless be obtained for the entire porous medium range, $H(\rho)=\rho^m/\prt{m-1}$ for $m>1$, by relaxing the assumption to $\lim_{\epsilon\to 0} \epsilon H''(\epsilon) = 0$. The proof, given in \cref{th:discreteHgradientestimateRelaxed}, handles the singularity by defining a regularised form of $K'$.
	}

	\revisionTwo{
	The theorem also assumes $H\in C^2([0,\infty))$, because our proof of compactness (\cref{th:timeshift,th:spaceshift} in particular) requires $H'$ to be locally Lipschitz. Independently of the previous discussion, this again excludes linear diffusion, as well as porous medium diffusion for $m<2$, all of which fail at the origin. Consequently, with the additional work in \cref{th:discreteHgradientestimateRelaxed}, the result in \cref{th:convergence} also applies to the power law case $m=2$, relevant in mathematical biology applications \cite{C.H.S2018,C.M.S+2019}.
	}

	\revisionTwo{
		The assumption on $H$ can be relaxed further if, in turn, the initial datum is strictly positive. Given $\rho^0>0$, \cref{th:aprioriLinfboundrho} guarantees the positivity of the entire solution. The proof is generalised to include linear diffusion and all the porous medium cases, bypassing the issues at the origin altogether. This extension is immediate for porous media, and the case of linear diffusion in detailed in \cref{sec:linear}.
	}

	The compactness proof in \cref{sec:compactness} relies on the local Lipschitz continuity of $H'$ as well as the bounds in \cref{th:aprioriLinfboundrho}. These bounds are general for the scheme with periodic boundary conditions, but require an additional assumption to handle the no-flux boundary conditions; namely, a sign on the normal spatial derivative of $\prt{V+W\conv\rho}$ at the boundary of $\Omega$. A non-negative (respectively non-positive) sign guarantees the upper (resp. lower) bound.
\end{remark}

\begin{remark}[Periodic boundary conditions]
	\revisionTwo{
		The analysis and numerical solution of \Cref{eq:continuous} are often performed in the setting of no-flux boundary conditions, as is done in this work. Nevertheless, all of our results can be immediately generalised to the case of periodic boundary conditions. This is because our analysis only exploits the lack of boundary terms when performing integration/summation by parts.
	}
\end{remark} \section{\textit{A Priori} Estimates} \label{sec:apriori_estimates}

The purpose of this section is threefold. We begin by proving existence of non-negative solutions to scheme \eqref{eq:discrete}, showing that the numerical method preserves the mass, and stating the discrete energy-dissipation property. Next, by mimicking estimate \eqref{eq:flowinterchange_continuous_formal} at the discrete level, we obtain discrete gradient information of $H'$, which, in return, will yield strong compactness of the density, $\rho\h$, itself. We conclude the section by establishing uniform $L^\infty$-bounds and a control from below of the solution.

\subsection{Existence of Solutions and Energy Dissipation Property} \label{sec:existence}

We proceed to show existence through a fixed-point argument.

\begin{theorem} [Existence of solutions] \label{thm:existence}
	Let $\rho_0(x)\in \Lone(\Omega)$ be a non-negative initial datum for \Cref{eq:continuous}. Then, there exist a unique, non-negative solution $\set{\rho\i\n}$, $(i,n)\in \curlyI\times\curlyN$, to scheme \eqref{eq:discrete}. Furthermore, the mass of the solution is conserved, i.e.,
	\begin{align}
		\psnorm{1}{\Omega}{\rho\h(t)} = \psnorm{1}{\Omega}{\rho_0}
	\end{align}
	for all $t\in [0,T]$.
\end{theorem}

\begin{proof}
	The existence proof is based on an application of Brouwer's fixed-point theorem. To this end, we consider the compact, convex set
	\begin{align}
		\curlyX\coloneqq\set*{\theta\in\R^{2M} \mid \theta\i\geq 0, \text{ for } i\in\curlyI, \text{ and } \sum_{i=1}^{2M} \theta\i\Dx \leq \psnorm{1}{\Omega}{\rho_0}},
	\end{align}
	define a function $\curlyG:\curlyX\to\R^{2M}$, and prove that this function is, indeed, a fixed-point operator.

	Suppose that the solution of scheme \eqref{eq:discrete}, $\set{\rho\i\n}_i$, is known at time $t\n$, for some $n\in\curlyN$. The function, $\curlyG$, is defined through the implicit relation $\tilde\theta = \curlyG\prt{\theta}$, where
	\begin{align} \label{eq:implicitrelation}
		\tilde\theta\i = \rho\i\n - \frac{\Dt}{\Dx} \prt{\tilde F\ih - \tilde F\imh},
	\end{align}
	for $i\in \curlyI$. The fluxes (which would coincide with the numerical fluxes at the fixed point) read
	\begin{align} \label{eq:implicitrelationflux}
		\tilde F\ih = \tilde \theta\i \pos{u\ih} + \tilde \theta\ip \neg{u\ih},
	\end{align}
	where $i=1,\ldots, 2M-1$, and $\tilde F_{1/2} = \tilde F_{2M+1/2} = 0$, to account for the boundary conditions. Note that the velocity terms are evaluated at the argument of the operator, $\theta$, rather than $\tilde\theta$:
	\begin{align}
		u\ih = - \frac{\xi\ip - \xi\i}{\Dx},
	\end{align}
	and $\xi\i = H'(\theta\i) + V_i + (W\conv \theta\nss)_i$, where $\theta\nss = \prt{\rho\n+\theta}/2$.

	It is easy to see that, as a composition of continuous functions, the operator $\curlyG$ itself is continuous \revisionTwo{(recall that we operate under the assumptions of \cref{th:convergence})}. Thus, it remains to show that it maps into the set $\curlyX$, i.e., that the image is non-negative and satisfies the mass constraint. We proceed by showing that $\tilde\theta$, the image of the operator, is non-negative, arguing by contradiction.

	Suppose there are certain values $i\in\curlyI$ such that $\tilde\theta\i <0$. Without loss of generality we assume that the corresponding cells $C\i$ form a single contiguous cluster. Therefore, we may assume that the negative values lie on a cluster of the form $j \leq i \leq k$, for some values $1 \leq j \leq k \leq 2M$.

	Summing \cref{eq:implicitrelation} over the pathological range and evaluating the flux terms as defined in \eqref{eq:implicitrelationflux} yields
	\begin{align} \label{eq:implicitrelationsummed}
		\revisionTwo{\sum_{i=j}^{k} \prt{\tilde\theta\i - \rho\i\n}\frac{\Dx}{\Dt} = - \tilde \theta\k \pos{u\kh} - \tilde \theta\kp \neg{u\kh} + \tilde \theta\jm \pos{u\jmh} + \tilde \theta\j \neg{u\jmh}.}
	\end{align}
	The left hand side is strictly negative, whereas all the right hand side terms are non-negative, yielding the desired contradiction.

	We stress that this argument applies also if several clusters of this type were to exist and if the cluster is degenerate, i.e., $j=k$.

	Having shown the non-negativity of the image, we note
	\begin{align} \label{eq:conservation_of_mass_in_fixedpt}
		\sum_{i=1}^{2M} \tilde\theta\i\Dx
		= \sum_{i=1}^{2M} \prt*{\rho\i\n - \frac{\Dt}{\Dx} \prt{\tilde F\ih - \tilde F\imh}}\Dx
		= \psnorm{1}{\Omega}{\rho\h(t\n)}
		= \psnorm{1}{\Omega}{\rho_0},
	\end{align}
	using \cref{eq:implicitrelation} and no-flux condition, $\tilde F\ih = \tilde F\imh=0$. As a consequence, $\curlyG$ maps the set $\curlyX$ into itself. Invoking this property in conjunction with the continuity of $\curlyG$, we may apply Brouwer's fixed-point theorem to infer the existence of a vector $\rho\np$ satisfying $\rho\np = \curlyG\prt{\rho\np}$. It is readily verified that the fixed point satisfies scheme \eqref{eq:discrete}, and that the conservation of mass follows from the same argument as in \cref{eq:conservation_of_mass_in_fixedpt}.
\end{proof}

Before proceeding with the analysis, we recall the aforementioned energy-dissipation property. As discussed in \cite{B.C.H2020}, the variational structure of \cref{eq:continuous} persists unconditionally in the discretised problem, thanks to the scheme, for a discrete analogue of the energy \eqref{eq:continuousEnergy}:
\begin{equation}
	\label{eq:discreteenergy}
	E_\Delta(\set{\rho\i\n}\i)=\sum_{i=1}^{2M}\prt*{ H(\rho\i^n)+ V\i\rho\i^n+\frac{1}{2}\sum_{j=1}^{2M}W_{i-j}\rho\i^n\rho_j^n\Dx}\Dx.
\end{equation}
The variation of this free energy over a single time step can be computed directly:
\begin{align}
	E_\Delta(\rho\np)-E_\Delta(\rho\n)
	 & = \sum_{i=1}^{2M}\left( H(\rho\i\np)+ V\i\rho\i\np+\frac{1}{2}\sum_{k=1}^{2M}\Wik\rho\i\np\rho\k\np\Dx \right)\Dx                 \\
	 & \quad -\sum_{i=1}^{2M}\left( H\prt{\rho\i\n}+ V\i\rho\i\n+\frac{1}{2}\sum_{k=1}^{2M}\Wik\rho\i\n\rho\k\n \Dx \right)\Dx           \\
	 & \leq \revisionTwo{\sum_{i=1}^{2M}\prt*{\prt*{H'(\rho\i\np) + V_i + \sum_{k=1}^{2M} \Wik \rho\k\nss\Dx} (\rho\i\np-\rho\i\n)}\Dx,}
\end{align}
having used the definition of $\rho\nss=(\rho\n + \rho\np)/2$ and the convexity of $H$. Substituting the scheme, we obtain
\begin{align}
	E_\Delta (\rho\np)-E_\Delta (\rho\n)
	\leq \revisionTwo{- \Dt \sum_{i=1}^{2M} \xi\i\np \frac{F\ih\np - F\imh\np}{\Dx}\Dx
		= - \Dt \sum_{i=1}^{2M-1} u\ih\np F\ih\np \Dx,}
\end{align}
having used the definition of the discrete entropy variables, $\xi\i\np$, the velocities $u\ih\np$, and a discrete integration by parts (see \cref{th:summationbyparts} in the Appendix). A straightforward simplification then yields
\begin{align}
	E_\Delta (\rho\np) - E_\Delta (\rho\n)
	 & \leq -\Dt\sum_{i=1}^{2M-1} \min(\rho\i\np,\rho\ip\np)\abs*{u\ih\np}^2
	\Dx                                                                                   \\
	 & = -\Dt\sum_{i=1}^{2M-1} \min(\rho\i\np,\rho\ip\np)\abs*{\ddx\xi\ih\n}^2\Dx \leq 0,
\end{align}
which is a discrete version of \cref{eq:energydissipation}.

\subsection{Discrete Gradient Estimate}

We now turn to the obtention of a discrete gradient estimate in the spirit of \cref{eq:flowinterchange_continuous_formal}. A crucial step in this endeavour is the construction of an auxiliary functional whose dissipation along solutions of \cref{eq:continuous} can be related to an $\Ltwo(\QT)$-bound on $\partial_x H'(\rho)$.

\begin{definition} [Auxiliary functional] \label{def:entropyconjugate}
	Let $H(\rho)$ be a convex density of internal energy \revisionTwo{such that $H''(s)>0$ for $s>0$, and $s^{-1}H''(s) \in L^1_{\mathrm{loc}}([0,\infty))$}. We define its associated auxiliary functional, $K(\rho)$, as any convex function which is bounded below by a constant $\cK$ and satisfies $\rho K''(\rho)=H''(\rho)$ for $\rho>0$.
\end{definition}

We remark that the convexity of $K$ follows from that of $H$. Furthermore, a choice which is bounded below can always be found by adding a linear function to a second primitive of $\rho^{-1}H''$, since convex functions lie above their tangents.

This associated functional can be used to define an average which generalises that of \cite{C.F.S2020}. It has also been used, in a more specific form, in the case of the Boltzmann entropy \cite{P.Z2018, C.C1970}.

\begin{lemma} [Generalised entropic average] \label{th:entropymean}
	Given a convex function, $H$, and its associated auxiliary functional, $K$, we define the \textit{generalised entropic average} of two positive real numbers $0 \leq x \leq y$ by
	\begin{align}
		\Emean{x}{y} \coloneqq \frac{H'(y)-H'(x)}{K'(y)-K'(x)},
	\end{align}
	whenever $x\neq y$, and by $\Emean{x}{y}=x$ otherwise. For any $0\leq x\leq y$, it holds
	\begin{align} \label{eq:entropymean}
		x \leq \Emean{x}{y} \leq y.
	\end{align}
\end{lemma}

\begin{proof}
	Without loss of generality, let $0 < x < y$. Using the definition of the auxiliary functional we find
	\begin{align}
		K'(y) - K'(x)
		= \int_x^y \frac{H''(s)}{s} \ds
		< \frac{1}{x} \int_x^y H''(s)\ds
		= \frac{1}{x} \prt{H'(y) - H'(x)}.
	\end{align}
	Upon rearranging, we deduce $x < \Emean{x}{y}$. Similarly,
	\begin{align}
		K'(y) - K'(x)
		= \int_x^y \frac{H''(s)}{s} \ds
		> \frac{1}{y} \int_x^y H''(s)\ds
		= \frac{1}{y} \prt{H'(y) - H'(x)},
	\end{align}
	which yields $y > \Emean{x}{y}$. \revisionTwo{The case $x=y$ is trivially true, and the case of $x=0$ nevertheless holds because of the assumption that $s^{-1}H''(s) \in L^1_{\mathrm{loc}}([0,\infty))$.}
\end{proof}

Before the gradient estimate, we show that the potential terms in the velocity can be controlled.

\begin{lemma} [$\Linf$-estimates of $\ddx V\h$ and $\ddx(W\conv\rho\nss)\h$] \label{th:potentialderivative}
	Given any solution, $\set{\rho\i\n}\i\n$, to scheme \eqref{eq:discrete}, there holds
	\begin{align}
		\sum_{i=1}^{2M-1} \prt*{\abs{\ddx V\ih}^2 + \abs{\ddx(W\conv\rho\nss)\ih}^2}\Dx \leq 2L \prt{\cVone}^2,
	\end{align}
	where $\cVone\coloneqq\psnorm{\infty}{\Omega}{V'} + \psnorm{\infty}{\Omega}{W'} \psnorm{1}{\Omega}{\rho_0}$, \revisionOne{and $2L$ is the length of the spatial domain where \cref{eq:continuous} is posed.}
\end{lemma}

\begin{proof}
	The confinement potential term is readily controlled by the derivative of $V$, since
	\begin{align}
		\sum_{i=1}^{2M-1} \abs{\ddx V\ih}^2 \Dx
		 & = \sum_{i=1}^{2M-1} \abs*{ \frac{V\ip-V\i}{\Dx} }^2 \Dx                            \\
		 & = \sum_{i=1}^{2M-1} \abs*{ \dashint_{C\i} \frac{ V(\Dx+s) - V(s)}{\Dx} \ds }^2 \Dx \\
		 & \leq \sum_{i=1}^{2M-1} \psnorm{\infty}{\Omega}{V'}^2 \Dx                           \\
		 & \leq 2L\psnorm{\infty}{\Omega}{V'}^2,
	\end{align}
	\revisionOne{having observed $\sum_{i=1}^{2M-1}\Dx\leq\sum_{i=1}^{2M}\Dx = 2L$.} The interaction potential may be bounded in a similar fashion:
	\begin{align}
		\sum_{i=1}^{2M-1} \abs{\ddx(W\conv\rho\nss)\ih}^2 \Dx
		 & = \sum_{i=1}^{2M-1} \abs*{ \frac{ (W\conv\rho\nss)\ip - (W\conv\rho\nss)\i }{\Dx} }^2 \Dx                                   \\
		 & = \sum_{i=1}^{2M-1} \abs*{ \sum_{j=1}^{2M} \frac{ W_{i+1-j} - W_{i-j} }{\Dx} \rho\j\nss \Dx }^2 \Dx                         \\
		 & = \sum_{i=1}^{2M-1} \abs*{ \sum_{j=1}^{2M} \dashint_{C\j} \frac{ W(x\ip - s) - W(x\i - s) }{\Dx} \ds \rho\j\nss \Dx }^2 \Dx \\
		 & = \psnorm{\infty}{\Omega}{W'} \sum_{i=1}^{2M-1} \abs*{ \sum_{j=1}^{2M} \rho\j\nss \Dx }^2 \Dx                               \\
		 & \leq 2L\psnorm{\infty}{\Omega}{W'}^2 \psnorm{1}{\Omega}{\rho_0}^2,
	\end{align}
	which concludes the proof.
\end{proof}

We are now ready to reproduce the estimate \eqref{eq:flowinterchange_continuous_formal} in order to obtain an $\Ltwo$-bound on the discrete gradient of $H'$.

\begin{proposition} [Discrete $\Ltwo(\Omega)$-estimate of $ \partial_x H'(\rho)$] \label{th:discreteHgradientestimate}
	Let $\set{\rho\i\n}$ be a solution to scheme \eqref{eq:discrete}. Then it holds
	\begin{align}
		\sum_{i=1}^{2M} \prt{K(\rho\i\np) - K(\rho\i\n)}\frac{\Dx}{\Dt} + \prt{1-\alpha} \psnormp{2}{\Omega}{\ddx H'(\rho\np)}
		\leq \alpha^{-1}L \prt{\cVone}^2
	\end{align}
	for any $\alpha\in\prt{0,1}$, where $\cVone$ is the constant from \cref{th:potentialderivative}.
\end{proposition}
\begin{proof}
	The convexity of $K$, together with scheme \eqref{eq:discrete}, yields
	\begin{align}
		\frac{K(\rho\i\np) - K(\rho\i\n)}{\Dt}
		\leq K'(\rho\i\np) \frac{\rho\i\np - \rho\i\n}{\Dt}
		= -K'(\rho\i\np) \ddx F\i\np,
	\end{align}
	which, upon summation over $i\in\curlyI$, gives
	\begin{align} \label{eq:summedestimatefirst}
		\sum_{i=1}^{2M} \frac{K(\rho\i\np) - K(\rho\i\n)}{\Dt}
		\leq - \sum_{i=1}^{2M} K'(\rho\i\np) \ddx F\i\np
		= \sum_{i=1}^{2M-1} \ddx K'(\rho\ih\np) F\ih\np.
	\end{align}
	The last equality follows from a discrete integration by parts (\cref{th:summationbyparts}) and the no-flux boundary conditions.

	Just as in the continuous estimate, cf. Eq. \eqref{eq:flowinterchange_continuous_formal}, we expand the flux term:
	\begin{align} \label{eq:fluxboundone}
		\begin{split}
			\ddx K'(\rho\ih\np) F\ih\np
			&= \ddx K'(\rho\ih\np) \brk*{ \rho\i\np\pos{u\ih\np} + \rho\ip\np\neg{u\ih\np}} \\
			&= \ddx K'(\rho\ih\np) { \tilde{\rho}\ih u\ih\np } \\
			&\quad + \ddx K'(\rho\ih\np) {\prt{\rho\ip\np-\tilde{\rho}\ih}\neg{u\ih\np} } \\
			&\quad +\ddx K'(\rho\ih\np) {\prt{\rho\i\np-\tilde{\rho}\ih}\pos{u\ih\np} }.
		\end{split}
	\end{align}
	Then, for any $\tilde{\rho}\ih\in\brk{\min\set{\rho\i\np,\rho\ip\np},\max\set{\rho\i\np,\rho\ip\np}}$, we have
	\begin{align}
		\ddx K'(\rho\ih\np) F\ih\np \leq\ddx K'(\rho\ih\np) \tilde{\rho}\ih u\ih\np,
	\end{align}
	using the convexity of the auxiliary functional, $K$. In particular, we may choose the generalised entropic average, $\tilde{\rho}\ih=\Emean{\rho\i\np}{\rho\ip\np}$, from \cref{th:entropymean}, and conclude
	\begin{align} \label{eq:fluxboundthree}
		\begin{split}
			\ddx K'(\rho\ih\np) F\ih\np
			&\leq \ddx K'(\rho\ih\np) \Emean{\rho\i\np}{\rho\ip\np} u\ih\np \\
			&= \ddx H'(\rho\ih\np)u\ih\np \\
			&= -\abs*{\ddx H'(\rho\ih\np)}^2+R\ih\np,
		\end{split}
	\end{align}
	using the definition of $\mu_{H}$. Here, the remainder is given by
	\begin{align}
		R\ih\np = - \ddx H'(\rho\ih\np)\prt*{\ddx V\ih + \ddx(W\conv\rho\nss)\ih}.
	\end{align}

	Substituting Eq. \eqref{eq:fluxboundthree} into Eq. \eqref{eq:summedestimatefirst}, we obtain
	\begin{align} \label{eq:boundwithremainder}
		\sum_{i=1}^{2M} \prt{K(\rho\i\np) - K(\rho\i\n)}\frac{\Dx}{\Dt}
		\leq -\sum_{i=1}^{2M-1} \abs{\ddx H'(\rho\ih\np)}^2\Dx +\sum_{i=1}^{2M-1}R\ih\np\Dx.
	\end{align}
	The remainder term can be estimated using the weighted Young's inequality:
	\begin{align} \label{eq:bounddVW}
		\begin{split}
			\sum_{i=1}^{2M-1} R\ih\np\Dx
			&= \sum_{i=1}^{2M-1} - \ddx H'(\rho\ih\np)\prt{\ddx V\ih + \ddx(W\conv\rho\nss)\ih} \Dx \\
			&\leq \sum_{i=1}^{2M-1}\prt*{ \alpha\abs{\ddx H'(\rho\ih\np)}^2 +\frac{1}{2\alpha}\abs{\ddx V\ih+\ddx(W\conv\rho\nss)\ih}^2}\Dx \\
			&\leq \alpha \sum_{i=1}^{2M-1} \abs{\ddx H'(\rho\ih\np)}^2 \Dx + \alpha^{-1} L \prt{\cVone}^2,
		\end{split}
	\end{align}
	using the bounds from \cref{th:potentialderivative}. Applying this to \cref{eq:boundwithremainder}, we finally obtain
	\begin{align}
		\sum_{i=1}^{2M} \prt{K(\rho\i\np) - K(\rho\i\n)}\frac{\Dx}{\Dt}
		\leq -\prt{1-\alpha} \sum_{i=1}^{2M-1} \abs{\ddx H'(\rho\ih\np)}^2\Dx +\alpha^{-1} L \prt{\cVone}^2,
	\end{align}
	which proves the statement.
\end{proof}

\begin{lemma}[Relaxed assumption on $H$]\label{th:discreteHgradientestimateRelaxed}
	\revisionTwo{
	The bound of Proposition \ref{th:discreteHgradientestimate},
	\begin{align}
		\sum_{i=1}^{2M} \prt{K(\rho\i\np) - K(\rho\i\n)}\frac{\Dx}{\Dt} + \prt{1-\alpha} \psnormp{2}{\Omega}{\ddx H'(\rho\np)}
		\leq \alpha^{-1}L \prt{\cVone}^2,
	\end{align}
	remains valid if the assumption $s^{-1} H''(s)\in L^1_{\mathrm{loc}}([0,\infty))$ is relaxed to
	$
		\lim_{\epsilon\to 0} \epsilon H''(\epsilon) = 0.
	$
	}
\end{lemma}
\begin{proof}
	\revisionTwo{
	In order to avoid the possible singularity at zero, we define the regularised auxiliary quantity, $K_\epsilon$, by
	\begin{align}
		(s+\epsilon) K''_\epsilon(s+\delta) = H''(s).
	\end{align}
	This allows us to define a regularised entropic mean, which is shown to satisfy
	\begin{align}
		x \leq \frac{H'(y) - H'(x)}{K'_\epsilon(y+\delta) - K'_\epsilon(x+\delta)} - \epsilon \leq y
	\end{align}
	in the same vein as the proof of Lemma \ref{th:entropymean}. To reproduce the estimate of \cref{th:discreteHgradientestimate}, we note that $K_\epsilon'$ remains bounded. Studying the dissipation of $K_\epsilon$ along the discrete solution, we find
	\begin{align}
		\lll \frac{1}{\Dt} \sum_{i=1}^{2M} (K_\epsilon(\rho\i\np + \delta )-K_\epsilon(\rho\i\n+ \delta))                                                                                                      \\
		 & \leq \sum_{i=1}^{2M} \ddx K_\epsilon'(\rho\ih\np + \delta) F\ih\np                                                                                                                                  \\
		 & \leq \sum_{i=1}^{2M} \ddx K_\epsilon'(\rho\ih\np + \delta) \left(\frac{H'(\rho\ip\np) - H'(\rho\i\np)}{K'_\epsilon(\rho\ip\np+\delta) - K'_\epsilon(\rho\i\np +\delta)} - \epsilon \right) u\ih\np,
	\end{align}
	cf. \cref{eq:fluxboundthree}. The first term of the parenthesis is estimated as in Proposition \ref{th:discreteHgradientestimate}. The second, which arises from the regularisation, vanishes in the limit, as we proceed to show. To this end, let us note that
	\begin{align}
		\label{eq:remainder_relaxed_condition}
		\abs*{\epsilon \ddx K_\epsilon'(\rho\ih\np + \delta) u\ih\np}
		 & \leq \curlyC \epsilon \prt*{K_\epsilon'(\max_{1\leq i\leq 2N} \rho\i\np+ \delta) + K_\epsilon'(\delta)},
	\end{align}
	where
	\begin{align}
		\curlyC = 4 \frac{{H'(\psnorm{1}{\Omega}{\rho_h}\Dx^{-1}) + \psnorm{\infty}{\Omega}{V} + \psnorm{\infty}{\Omega}{W}\psnorm{1}{\Omega}{\rho_h} }}{\Dx^2}.
	\end{align}
	While this constant depends rather poorly on $\Dx$, we remark that, for our purpose, it is sufficient to bound the quantity in Eq. \eqref{eq:remainder_relaxed_condition} for fixed $\Dx, \Dt$, and pass to the limit $\epsilon\to 0$. The first term on the right-hand side of the inequality is bounded, and thus vanishes in the limit. To control the second term, we recall that
	\begin{align}
		\epsilon K_\epsilon'(s + \delta) = \epsilon K_\epsilon'(1 + \delta) + \epsilon \int_1^s\frac{H''(t)}{t+\epsilon}\mathrm{d} t,
	\end{align}
	where the first term on the right-hand side vanishes in the limit.
	Choosing $\delta = s = \epsilon/2$, we find
	\begin{align}
		\lim_{\epsilon \to 0} \epsilon \abs{K_\epsilon'(\epsilon)} = \lim_{\epsilon \to 0} \epsilon \int_{\epsilon/2}^1 \frac{H''(t)}{t+\epsilon}\mathrm{d}t.
	\end{align}
	Under the original assumption, $s^{-1}H''(s)\in L^1_{\mathrm{loc}}([0,\infty))$, the limit is zero. In the relaxed setting, where the integrand might not integrable, we nevertheless write
	\begin{align}
		\lim_{\epsilon \to 0} \epsilon \abs{K_\epsilon'(\epsilon)} = \lim_{\epsilon\to 0} \frac{\int_{\epsilon/2}^1 \frac{H''(t)}{t+\epsilon}\mathrm{d}t}{1/\epsilon} = \lim_{\epsilon \to 0 } \frac{-2\frac{H''(\epsilon/2)}{3\epsilon}}{-1/\epsilon^2} = \frac43 \lim_{\epsilon\to 0}\frac\epsilon2 H''(\epsilon/2) = 0
	\end{align}
	using L'H\^{o}pital's rule and the relaxed assumption, which completes the proof.
	}
\end{proof}

\begin{corollary} [Discrete $\Ltwo(\OmegaT)$-bound of $\partial_x H'(\rho)$] \label{th:aprioriL2boundgradH}
	Let $\set{\rho\i\n}$ be a solution to scheme \eqref{eq:discrete}. Then there exists a constant $\cdH>0$, independent of the mesh size, such that
	\begin{align}
		\psnormp{2}{\OmegaT}{\ddx H'(\rho\h)} \leq \cdH.
	\end{align}
\end{corollary}

\begin{proof}
	Performing discrete time integration of the result of \cref{th:discreteHgradientestimate} yields
	\begin{align}
		\sum_{i=1}^{2M} \prt{K(\rho\i^N) - K(\rho\i^0)}\Dx + \prt{1-\alpha}\psnormp{2}{\OmegaT}{\ddx H'(\rho\h)}
		\leq \alpha^{-1} L \prt{\cVone}^2 T.
	\end{align}
	Using $\mathcal{C}_K$, the lower bound on $K$ (see \cref{def:entropyconjugate}), we obtain
	\begin{align}
		\prt{1-\alpha} \psnormp{2}{\OmegaT}{\ddx H'(\rho\h)}
		\leq \alpha^{-1} L \prt{\cVone}^2 T +2L(-\mathcal{C}_K)^+ +\psnorm{1}{\Omega}{K(\rho^0)}.
	\end{align}
	Choosing $\alpha\in\prt{0,1}$ concludes the proof.
\end{proof}
 \subsection{Upper and Lower Bounds}

We now turn our attention to deriving bounds to control the approximate solution $\rho\h$ above and below in terms of the initial datum.

\begin{proposition}[Upper and lower bounds for $\rho\h$]\label{th:aprioriLinfboundrho}
	Let $\set{\rho\i\n}$ be a solution to scheme \eqref{eq:discrete}. Then:
	\begin{enumerate}[(i)]
		\item $\displaystyle \max_{i\in\curlyI} \rho\i\n \leq \prt*{\frac{1}{1 - \Dt{\curlyC}_V^{(2)}}}\n \max_{i\in\curlyI}\rho\i^0$, provided $\Dt {\curlyC}_V^{(2)} <1$;
		\item $\displaystyle \min_{i\in\curlyI} \rho\i\n \geq \prt*{\frac{1}{1 + \Dt{\curlyC}_V^{(2)}}}\n \min_{i\in\curlyI}\rho\i^0$;
	\end{enumerate}
	where $\cVtwo = \psnorm{\infty}{\Omega}{V''} + \psnorm{\infty}{\Omega}{W''} \psnorm{1}{\Omega}{\rho_0}$.
	In particular, $\rho\h \in \Linf(\QT)$, with some uniform bound $\cinf$.
\end{proposition}

\begin{proof}
	We will consider the solution to the scheme at two subsequent time instances, $t\n$ and $t\np$, and compare them. Let
	$i_0 \in \argmax\limits_{i\in \curlyI} \set{\rho\i\np}$,
	and note that $ \rho\i\np \leq \rho_{i_0}\np $ for all $i\in \curlyI$. The scheme yields
	\begin{align}
		\rho_{i_0}\np = \rho_{i_0}\n - \Dt \frac{F_{i_0+1/2}\np-F_{i_0-1/2}\np}{\Dx}.
	\end{align}
	Rewriting the numerical flux at the right cell interface, we observe
	\begin{align}
		F_{i_0+1/2}\np
		 & = \rho_{i_0}\np \pos{u_{i_0+1/2}\np} + \rho_{i_0+1}\np \neg{u_{i_0+1/2}\np}               \\
		 & = \rho_{i_0}\np u_{i_0+1/2}\np + \prt{\rho_{i_0+1}\np-\rho_{i_0}\np} \neg{u_{i_0+1/2}\np} \\
		 & \geq \rho_{i_0}\np u_{i_0+1/2}\np,
	\end{align}
	having used $\rho_{i_0+1}\np\leq\rho_{i_0}\np$. We recall the definition of scheme \eqref{eq:discrete} and split the velocity term $u\ih\np$ into the entropic part and the drift part as defined in \eqref{eq:velocity_split_entropic_potential}.
	This leads to
	\begin{align}
		F_{i_0+1/2}\np
		 & \geq \rho_{i_0}\np u_{i_0+1/2}\np                                                                           \\
		 & = \rho_{i_0}\np h_{i_0+1/2}\np + \rho_{i_0}\np v_{i_0+1/2}\np                                               \\
		 & = -\rho_{i_0}\np \frac{H'\prt{\rho_{i_0+1}\np} - H'\prt{\rho_{i_0}\np}}{\Dx} + \rho_{i_0}\np v_{i_0+1/2}\np \\
		 & \geq \rho_{i_0}\np v_{i_0+1/2}\np,
	\end{align}
	having, once again, used $\rho_{i_0+1}\np \leq \rho_{i_0}\np$, in conjunction with the monotonicity of $H'$. An analogous computation readily shows that $F_{i_0-1/2}\np \leq \rho_{i_0}\np v_{i_0-1/2}\np$. Finally, we note that
	\begin{align}
		\psnorm{\infty}{\Omega}{\rho\h\np}
		= \rho_{i_0}\np
		 & = \rho_{i_0}\n - \Dt \frac{F_{i_0+1/2}\np - F_{i_0-1/2}\np}{\Dx}                                  \\
		 & \leq \rho_{i_0}\n + \rho_{i_0}\np \frac{\Dt}{\Dx}\prt*{v_{i_0-1/2}\np - v_{i_0+1/2}\np}           \\
		 & \leq \psnorm{\infty}{\Omega}{\rho\h\n} + \Dt {\curlyC}_V^{(2)}\psnorm{\infty}{\Omega}{\rho\h\np},
	\end{align}
	having found $\psnorm{\infty}{\Omega}{\ddx v\i\np} \leq {\curlyC}_V^{(2)}$ through the same technique employed in the proof of \cref{th:potentialderivative}; we obtain
	\begin{align} \label{eq:linffirstinequality}
		\psnorm{\infty}{\Omega}{\rho\h\np}
		\leq \frac{1}{1-\Dt{\curlyC}_V^{(2)}} \psnorm{\infty}{\Omega}{\rho\h\n},
	\end{align}
	assuming $\Dt{\curlyC}_V^{(2)} < 1$. The repeated application of \cref{eq:linffirstinequality} yields
	\begin{align} \label{eq:linfsecondinequality}
		\psnorm{\infty}{\Omega}{\rho\h\n}
		\leq \frac{1}{\prt*{1-\Dt {\curlyC}_V^{(2)}}^{n}} \psnorm{\infty}{\Omega}{\rho\h^0},
	\end{align}
	which concludes the proof of the first bound.

	The second bound is proven in a parallel way. Let
	$i_0 \in \argmin_{i\in\curlyI} \set{\rho\i\np}$.
	We observe
	\begin{align}
		F_{i_0+1/2}\np
		 & = \rho_{i_0}\np \pos{u_{i_0+1/2}\np} + \rho_{i_0+1}\np \neg{u_{i_0+1/2}\np}               \\
		 & = \rho_{i_0}\np u_{i_0+1/2}\np + \prt{\rho_{i_0+1}\np-\rho_{i_0}\np} \neg{u_{i_0+1/2}\np} \\
		 & \leq \rho_{i_0}\np u_{i_0+1/2}\np,
	\end{align}
	and, through the splitting of the velocity, find
	\begin{align}
		F\np_{i_0+1/2} \leq \rho_{i_0}\np u_{i_0+1/2}\np \leq \rho_{i_0}\np v_{i_0+1/2}\np,
	\end{align}
	having used the convexity of $H$. Estimating $F\np_{i_0-1/2}$ in a similar fashion, we obtain
	\begin{align}
		\rho_{i_0}\np \geq \rho_{i_0}\n - \rho_{i_0}\np \, \frac{\Dt}{\Dx} \prt*{v_{i_0+1/2}\np - v_{i_0-1/2}\np} \geq \rho_{i_0}\n - \rho_{i_0}\np {\curlyC}_V^{(2)} \Dt,
	\end{align}
	with ${\curlyC}_V^{(2)}>0$ as above. Rearranging and iterating the inequality yields the statement and concludes the proof.
\end{proof}
 \section{Compactness} \label{sec:compactness}

Equipped with the \term{a priori} estimates obtained in the preceding section, we are now ready to establish the weak compactness of the velocity terms as well as the strong compactness of the density, $\rho_h$. To this end, we first prove the strong compactness of the entropic term, $H'\prt{\rho\h}$.

\begin{proposition} [$\Ltwo\prt{\QT}$-compactness of $H'\prt{\rho\h}$] \label{thm:compactness_entropic_part}
	Let $\prt{\rho\h}$ be a sequence of solutions to scheme \eqref{eq:discrete} for decreasing $h$. Then, the family $\prt{H'\prt{\rho\h}}$ converges strongly (up to a subsequence) in the $\Ltwo\prt{\QT}$ sense to some function $\chi\in\Ltwo\prt{\QT}$.
\end{proposition}

\cref{thm:compactness_entropic_part} can be obtained by invoking the well-known $\Lp$-compactness criterion of Kolmogorov-Riesz-Fr\'{e}chet \cite[Theorem 4.26]{Brezis2010}. This involves establishing a control on the time and space shifts of $H'\prt{\rho\h}$.

\begin{lemma} [Time translates of $H'(\rho\h)$] \label{th:timeshift}
	Let $\set{\rho\i\n}$ be a solution to scheme \eqref{eq:discrete}. Then, \revisionTwo{for any $\varepsilon>0$,} there exists a constant $\mathcal{C}>0$, \revisionTwo{uniform in $\Dx$ and in $\Dt \leq 1/{\curlyC}_V^{(2)} - \varepsilon$}, such that
	\begin{align}
		\int_{0}^{T-\tau} \int_{\Omega} \abs{H'\prt{\rho\h(t+\tau, x)}-H'\prt{\rho\h(t, x)}}^2 \dx\dt
		\leq \mathcal{C}\tau.
	\end{align}
\end{lemma}

\begin{proof}
	Let $\tau>0$ be given. Then there exists a unique integer, $K\in\curlyN$, such that
	\begin{align}
		K \Dt \leq \tau \leq (K+1) \Dt.
	\end{align}
	We introduce the notation $\Dl, \Dh>0$ in order to write
	\begin{align}
		\tau = K \Dt + \Dl = (K+1)\Dt - \Dh,
	\end{align}
	and $\Dt = \Dl + \Dh$. As a consequence, we have $T-\tau \in I^{N-K}$.

	\begin{figure}[ht]
		\centering
		\tikzset{
	hatch distance/.store in=\hatchdistance,
	hatch distance=10pt,
	hatch thickness/.store in=\hatchthickness,
	hatch thickness=2pt
}

\begin{tikzpicture}[scale=0.95]
	\def\Linterval{1.3}
	\def\Ltau{4.45}
	\def\Lboundarylow{0.35}
	\def\Lboundaryhigh{0.55}
	\def\Lseparation{4.0}
	\def\Ldiscretisation{0.8}
	\def\Llabelx{0.95}
	\def\Llabely{-0.75}
	\def\Llabels{0.1}

\draw[white, thin, pattern=flexible hatch, hatch distance=5pt, hatch thickness=0.25pt, pattern color=black!50!white]
	({0*\Linterval},{-\Lboundarylow}) rectangle
	({\Ltau+0*\Linterval},{\Lboundarylow});

	\draw[white, thin, pattern=flexible hatch, hatch distance=5pt, hatch thickness=0.25pt, pattern color=black!50!white]
	({7*\Linterval},{-\Lboundarylow-\Lseparation}) rectangle
	({\Ltau+7*\Linterval},{\Lboundarylow-\Lseparation});

\draw[thick] ({0*\Linterval},{0}) -- ({7*\Linterval},{0});
	\draw[thick] ({\Ltau+0*\Linterval},{-\Lseparation}) -- ({\Ltau+7*\Linterval},{-\Lseparation});

\foreach \i in {0,...,7}{
			\draw[thick]
			({\i*\Linterval},{-\Lboundarylow}) --
			({\i*\Linterval},{\Lboundaryhigh});

			\draw[thick]
			({\Ltau+\i*\Linterval},{-\Lboundaryhigh-\Lseparation}) --
			({\Ltau+\i*\Linterval},{\Lboundarylow-\Lseparation});
		}

\foreach \i in {-3,...,3}{
			\draw[thick, dotted]
			({\Ltau+\i*\Linterval},{-\Lboundarylow}) --
			({\Ltau+\i*\Linterval},{\Lboundarylow});
		}
	\foreach \i in {4,...,10}{
			\draw[thick, dotted]
			({\i*\Linterval},{-\Lboundarylow-\Lseparation}) --
			({\i*\Linterval},{\Lboundarylow-\Lseparation});
		}

\node[above] at
	({0*\Linterval},{\Lboundaryhigh})
	{$t^{0}=0$};
	\node[below] at
	({\Ltau+0*\Linterval},{-\Lboundaryhigh-\Lseparation})
	{$t^{0}=0$};

	\node[above] at
	({1*\Linterval},{\Lboundaryhigh})
	{$t^{1}$};
	\node[below] at
	({\Ltau+1*\Linterval},{-\Lboundaryhigh-\Lseparation})
	{$t^{1}$};

	\node[above] at
	({2*\Linterval},{\Lboundaryhigh})
	{$\cdots$};
	\node[below] at
	({\Ltau+2*\Linterval},{-\Lboundaryhigh-\Lseparation})
	{$\cdots$};

	\node[above] at
	({7*\Linterval},{\Lboundaryhigh})
	{$t^{N}=T$};
	\node[below] at
	({\Ltau+7*\Linterval},{-\Lboundaryhigh-\Lseparation})
	{$t^{N}=T$};

	\node[above] at
	({6*\Linterval},{\Lboundaryhigh})
	{$t^{N-1}$};
	\node[below] at
	({\Ltau+6*\Linterval},{-\Lboundaryhigh-\Lseparation})
	{$t^{N-1}$};

\node[below] at
	({\Ltau + 0*\Linterval},{-\Lboundarylow})
	{$\tau$};
	\node[above] at
	({-\Ltau + \Ltau + 7*\Linterval},{+\Lboundarylow-\Lseparation})
	{$T-\tau$};

\draw[<-]
	({7*\Linterval+0.5},{0}) --
	({7*\Linterval+2.5},{0})
	node[fill=white] {$\rho\h\prt{t+\tau}$};

	\draw[<-]
	({\Ltau-0.5},{-\Lseparation}) --
	({\Ltau-2.5},{-\Lseparation})
	node[fill=white] {$\rho\h\prt{t}$};

\draw[<->]
	({\Linterval*5},{\Lboundaryhigh+\Ldiscretisation}) --
	({\Linterval*6},{\Lboundaryhigh+\Ldiscretisation})
	node[midway, above] {$\Dt$};

	\draw[<->]
	({\Linterval*3},{\Lboundaryhigh+\Ldiscretisation}) --
	({\Ltau+\Linterval*0},{\Lboundaryhigh+\Ldiscretisation})
	node[midway, above] {$\Dl$};

	\draw[<->]
	({\Ltau+\Linterval*0},{\Lboundaryhigh+\Ldiscretisation}) --
	({\Linterval*4},{\Lboundaryhigh+\Ldiscretisation})
	node[midway, above] {$\Dh$};

\draw[<-]
	({(\Ltau+\Linterval*8)/2},{-\Lseparation+\Lboundarylow-\Llabels}) --
	({(\Ltau+\Linterval*8)/2-\Llabelx},{\Llabely-\Lseparation/2});
	\draw[<-]
	({(\Ltau+\Linterval*8)/2},{-\Lboundarylow+\Llabels}) --
	({(\Ltau+\Linterval*8)/2-\Llabelx},{\Llabely-\Lseparation/2})
	node[left] {$\abs{H'\prt{\rho\i^{n+K}}-H'\prt{\rho\i^{n}}}^2\Dh$};

	\draw[<-]
	({(\Ltau+\Linterval*9)/2},{-\Lseparation+\Lboundarylow-\Llabels}) --
	({(\Ltau+\Linterval*9)/2+\Llabelx},{-\Llabely-\Lseparation/2});
	\draw[<-]
	({(\Ltau+\Linterval*9)/2},{-\Lboundarylow+\Llabels}) --
	({(\Ltau+\Linterval*9)/2+\Llabelx},{-\Llabely-\Lseparation/2})
	node[right] {$\abs{H'\prt{\rho\i^{n+K+1}}-H'\prt{\rho\i^{n}}}^2\Dl$};

	\draw[<-]
	({\Linterval*3},{-\Lboundaryhigh+\Llabels}) --
	({\Linterval*2.15)},{\Llabely-\Lseparation/10})
	node[below] {$K\Dt$};
	\draw[<-]
	({\Linterval*4},{-\Lboundaryhigh+\Llabels}) --
	({\Linterval*3.50)},{\Llabely-\Lseparation/10})
	node[below] {$\prt{K+1}\Dt$};
\end{tikzpicture} 		\caption{Split discretisation of \cref{eq:cmptimeshift}. The integrand becomes $\abs{H'\prt{\rho\i\K}-H'\prt{\rho\i\n}}^2$ over the $\Dh$ segments; in the $\Dl$ sections it becomes $\abs{H'\prt{\rho\i\Kp}-H'\prt{\rho\i\n}}^2$ instead.}
		\label{fig:dialintegral}
	\end{figure}

	The splitting induces the following space discretisation of the time translates:
	\begin{align} \label{eq:cmptimeshift}
		\int_{0}^{T-\tau} \int_{\Omega} \abs{H'\prt{\rho\h(t+\tau, x)}-H'\prt{\rho\h(t, x)}}^2 \dx\dt = I_{\Dh} + I_{\Dl},
	\end{align}
	where
	\begin{align}
		I_{\Dh} \coloneqq \sum_{n=0}^{N-K} \sum_{i=1}^{2M} \abs{H'\prt{\rho\i\K}-H'\prt{\rho\i\n}}^2 \Dx \Dh,
	\end{align}
	as well as
	\begin{align}
		I_{\Dl} \coloneqq \sum_{n=0}^{N-K-1} \sum_{i=1}^{2M} \abs{H'\prt{\rho\i\Kp}-H'\prt{\rho\i\n}}^2 \Dx \Dl,
	\end{align}
	see \cref{fig:dialintegral} for detail.

	The first double sum is either $0$ (when $\tau<\Dt$, which gives $K=0$) or, using the Lipschitz continuity of $H'$ with constant $\cH$,
	\begin{align}
		I_{\Dh} & =
		\sum_{n=0}^{N-K}
		\sum_{i=1}^{2M}
		\abs{H'\prt{\rho\i\K}-H'\prt{\rho\i\n}}^2
		\Dx\Dh                  \\
		        & \leq
		\curlyC_{H'}
		\sum_{n=0}^{N-K}
		\sum_{i=1}^{2M}
		\prt{H'\prt{\rho\i\K}-H'\prt{\rho\i\n}}
		\prt{\rho\i\K-\rho\i\n}
		\Dx\Dh                  \\
		        & =\curlyC_{H'}
		\sum_{n=0}^{N-K}
		\sum_{i=1}^{2M}
		\prt{H'\prt{\rho\i\K}-H'\prt{\rho\i\n}}
		\sum_{k=1}^{K}
		\frac{\rho\i\nk-\rho\i\nkm}{\Dt}
		\Dt\Dx\Dh               \\
		        & =
		-\curlyC_{H'}
		\sum_{n=0}^{N-K}
		\sum_{i=1}^{2M}
		\prt{H'\prt{\rho\i\K}-H'\prt{\rho\i\n}}
		\sum_{k=1}^{K}
		\ddx F\i\nk
		\Dt\Dx\Dh,
	\end{align}
	having used scheme \eqref{eq:discrete} in the last line. \revisionTwo{Note that $\curlyC_{H'}$ is not the same as $\cdH$, the $\Ltwo$ bound on the discrete gradient of $H'\prt{\rho\h}$ from  \cref{th:aprioriL2boundgradH}.} Upon summation by parts and rearranging the sums we obtain
	\begin{align}
		I_{\Dh}=
		\curlyC_{H'}
		\sum_{k=1}^{K}
		\sum_{n=0}^{N-K}
		\sum_{i=1}^{2M-1}
		\prt{\ddx H'\prt{\rho\ih\K}-\ddx H'\prt{\rho\ih\n}}
		F\ih\nk
		\Dx\Dt\Dh,
	\end{align}
	which can be controlled by substituting the flux and using the $\Linf$-bound from \cref{th:aprioriLinfboundrho}:
	\begin{align}
		\lll
		\sum_{k=1}^{K}
		\sum_{n=0}^{N-K}
		\sum_{i=1}^{2M-1}
		\prt*{\ddx H'\prt{\rho\ih\K}-\ddx H'\prt{\rho\ih\n}}
		\brk*{\rho\i\nk\pos{u\ih\nk} + \rho\ip\nk\neg{u\ih\nk}}
		\Dx\Dt\Dh
		\\
		 & \leq
		\curlyC_\infty
		\sum_{k=1}^{K}
		\sum_{n=0}^{N-K}
		\sum_{i=1}^{2M-1}
		\prt*{\abs*{\ddx H'\prt{\rho\ih\K}}+\abs*{\ddx H'\prt{\rho\ih\n}}}
		\abs*{u\ih\nk}
		\Dx\Dt\Dh.
	\end{align}
	Thus we may write
	\begin{align} \label{eq:leftmostrightmost}
		I_{\Dh} \leq \curlyC_{H'}
		\curlyC_\infty
		\sum_{k=1}^{K} \prt*{I_{\Dh}^{(1)}(k) + I_{\Dh}^{(2)}(k)} \Dh,
	\end{align}
	where we introduced the notation
	\begin{align}
		I_{\Dh}^{(1)}(k) \coloneqq \sum_{n=0}^{N-K}
		\sum_{i=1}^{2M-1} \abs*{\ddx H'\prt{\rho\ih\K}}
		\abs*{u\ih\nk} \Dx \Dt,
	\end{align}
	and
	\begin{align}
		I_{\Dh}^{(2)}(k) \coloneqq \sum_{n=0}^{N-K}
		\sum_{i=1}^{2M-1}
		\abs*{\ddx H'\prt{\rho\ih\n}}
		\abs*{u\ih\nk} \Dx \Dt.
	\end{align}{}

	We begin by bounding the term, $I_{\Dh}^{(1)}(k)$, by splitting the velocity as in \eqref{eq:velocity_split_entropic_potential}:
	\begin{align}
		I_{\Dh}^{(1)}(k)
		 & =
		\sum_{n=0}^{N-K} \sum_{i=1}^{2M-1} \abs{h\ih^{n+K}} \abs{u\ih\nk} \Dx\Dt \\
		 & \leq
		\sum_{n=0}^{N-K} \sum_{i=1}^{2M-1}
		\abs{h\ih^{n+K}}
		\left( \abs{h\ih\nk}
		+ \abs{v\ih^{n+k}} \right)
		\Dx\Dt
		\\
		 & \leq
		\sum_{n=0}^{N-K} \sum_{i=1}^{2M-1}
		\left( \abs{h\ih^{n+K}}^2
		+ \frac{1}{2}\abs{h\ih^{n+k}}^2
		+\frac{1}{2}\abs{v\ih^{n+k}}^2 \right) \Dx\Dt,
	\end{align}
	using Young's inequality on each term. Extending the sum over $n$ to the set $\curlyN$ permits the combination of the terms involving $H'$,
	\begin{align} \label{eq:I_Dh_1_k_is_bounded}
		\begin{split}
			I_{\Dh}^{(1)}(k)
			&\leq
			\sum_{n=0}^{N} \sum_{i=1}^{2M-1}
			\left( \frac{3}{2}\abs{h\ih\n}^2 + \frac{1}{2}\abs{v\ih\n}^2 \right)\Dx\Dt
			\leq
			\frac{3}{2}\curlyC_{\partial_x H'} + LT \curlyC_{V}^{(1)},
		\end{split}
	\end{align}
	where the last line follows using the bounds of \cref{th:potentialderivative}.

	Treating the second term of Eq. \eqref{eq:leftmostrightmost}, $I_{\Dh}^{(2)}(k)$, in an identical fashion, we may use its bound in conjunction with that of Eq. \eqref{eq:I_Dh_1_k_is_bounded} in Eq. \eqref{eq:leftmostrightmost} to obtain a bound on the total contribution:
	\begin{align} \label{eq:bound_on_I_Dh}
		\begin{split}
			I_{\Dh}
			&=
			\sum_{n=0}^{N-K}
			\sum_{i=1}^{2M}
			\abs{H'\prt{\rho\i\K}-H'\prt{\rho\i\n}}^2
			\Dx\Dh\\
			&\leq \curlyC_{H'}
			\curlyC_\infty
			\sum_{k=1}^{K} \prt*{I_{\Dh}^{(1)}(k) + I_{\Dh}^{(2)}(k)} \Dh\leq
			\mathcal{C}K\Dh,
		\end{split}
	\end{align}
	where
	$
		\mathcal{C}\coloneqq\curlyC_{H'}
		\curlyC_\infty
		\left(
		3 \curlyC_{\partial_x H'}
		+2 LT \curlyC_{V}^{(1)}
		\right).
	$

	Using a similar argument, we find
	\begin{align} \label{eq:bound_on_I_Dl}
		\begin{split}
			I_{\Dl} &=
			\sum_{n=0}^{N-K-1}
			\sum_{i=1}^{2M}
			\abs{H'\prt{\rho\i\Kp}-H'\prt{\rho\i\n}}^2
			\Dx\Dl
			\leq
			\mathcal{C}\prt{K+1}\Dl.
		\end{split}
	\end{align}
	Substituting \cref{eq:bound_on_I_Dh} and \cref{eq:bound_on_I_Dl} into \cref{eq:cmptimeshift}, we obtain
	\begin{align}
		 &
		\int_{0}^{T-\tau}
		\int_{\Omega}
		\abs{H'\prt{\rho\h(t+\tau, x)}-H'\prt{\rho\h(t, x)}}^2
		\dx\dt
		\leq
		\mathcal{C}
		\brk*{K\Dh + \prt{K+1}\Dl}
		=
		\mathcal{C}\tau,
	\end{align}
	where the last equality holds because $\Dt=\Dl+\Dh$ and $K\Dt+\Dl = \tau$.
\end{proof}

\begin{lemma} [Space translate of $H'(\rho\h)$] \label{th:spaceshift}
	Let $\set{\rho\i\n}$ be a solution to scheme \eqref{eq:discrete}. Then, \revisionTwo{for any $\varepsilon>0$,} there exists a constant $\mathcal{C}>0$, \revisionTwo{uniform in $\Dx$ and in $\Dt \leq 1/{\curlyC}_V^{(2)} - \varepsilon$}, such that
	\begin{align}
		\int_{0}^{T} \int_{-L}^{L-z}
		\abs{H'\prt{\rho\h(t, x+z)}-H'\prt{\rho\h(t, x)}}^2 \dx\dt
		\leq
		\mathcal{C} z.
	\end{align}
\end{lemma}

\begin{proof}
	Fix $z>0$, and suppose $L-z \in C_{2M-K}$.
	Repeating the argument of the proof of \cref{th:timeshift}, we see $K=\floor*{\frac{z}{\Dx}}$ and $K\Dx \leq z \leq (K+1)\Dx$. Once again we recover a partition, this time of the spatial interval, $\Dx = \Dl+\Dh$; we have $\Dl=z - K\Dx$ and $\Dh=\prt{K+1}\Dx - z$. This splitting yields the discretisation
	\begin{align} \label{eq:spaceshiftssums}
		\int_{0}^{T}
		\int_{-L}^{L-z}
		\abs{H'\prt{\rho\h(t, x+z)}-H'\prt{\rho\h(t, x)}}^2
		\dx\dt = I_{\Dh} + I_{\Dl},
	\end{align}
	where
	\begin{align}
		I_{\Dh} \coloneqq
		\sum_{n=0}^{N}
		\sum_{i=1}^{2M-K}
		\abs{H'\prt{\rho\iK\n}-H'\prt{\rho\i\n}}^2
		\Dh
		\Dt,
	\end{align}
	and
	\begin{align}
		I_{\Dl} \coloneqq \sum_{n=0}^{N}
		\sum_{i=1}^{2M-K-1}
		\abs{H'\prt{\rho\iKp\n}-H'\prt{\rho\i\n}}^2
		\Dl
		\Dt.
	\end{align}{}

	To bound the term $I_{\Dl}$ in \cref{eq:spaceshiftssums}, we use the Lipschitz continuity of $H'$:
	\begin{align}
		I_{\Dl}
		 & =\sum_{n=0}^{N}
		\sum_{i=1}^{2M-K}
		\abs{H'\prt{\rho\iK\n}-H'\prt{\rho\i\n}}^2
		\Dh\Dt
		\\
		 & \leq
		\curlyC_{H'}
		\sum_{n=0}^{N}
		\sum_{i=1}^{2M-K}
		\prt{\rho\iK\n-\rho\i\n}
		\prt{H'\prt{\rho\iK\n}-H'\prt{\rho\i\n}}
		\Dh\Dt
		\\
		 & =
		\curlyC_{H'}
		\sum_{n=0}^{N}
		\sum_{i=1}^{2M-K}
		\prt{\rho\iK\n-\rho\i\n}
		\sum_{k=1}^{K}
		\frac{H'\prt{\rho\ik\n}-H'\prt{\rho\ikm\n}}{\Dx}
		\Dx\Dh\Dt
		\\
		 & \leq
		\curlyC_{H'}
		\sum_{n=0}^{N}
		\sum_{i=1}^{2M-K}
		\sum_{k=1}^{K}
		\prt{\rho\iK\n-\rho\i\n}
		\ddx H'\prt{\rho\ikmh\n}
		\Dx\Dh\Dt.
	\end{align}
	This expression can be controlled in terms of the estimates from \cref{th:aprioriL2boundgradH,th:aprioriLinfboundrho}:
	\begin{align}
		\label{eq:space_shifts_Dl_bound}
		\begin{split}
			I_{\Dl}
			&\leq
			2\curlyC_{H'}
			\curlyC_\infty
			\sum_{n=0}^{N}
			\sum_{i=1}^{2M-K}
			\sum_{k=1}^{K}
			\abs{ \ddx H'\prt{\rho\ikmh\n} }
			\Dx\Dh\Dt
			\\
			&\leq
			2 \curlyC_{H'} \sqrt{2LT}
			\curlyC_\infty
			\psnorm{2}{\QT}{\ddx H'\prt{\rho\h}}
			K \Dh
			\leq
			\mathcal{C} K \Dh.
		\end{split}
	\end{align}

	The second term in \cref{eq:spaceshiftssums}, $I_{\Dh}$, is controlled similarly. In fact,
	\begin{align} \label{eq:space_shifts_Dh_bound}
		I_{\Dh}
		=
		\sum_{n=0}^{N}
		\sum_{i=1}^{2M-K-1}
		\abs{H'\prt{\rho\iKp\n}-H'\prt{\rho\i\n}}^2
		\Dl
		\Dt
		\leq
		\mathcal{C}
		\prt{K+1}\Dl.
	\end{align}
	The combination of \cref{eq:space_shifts_Dl_bound,eq:space_shifts_Dh_bound} provides bound
	\begin{align}
		\int_{0}^{T} \int_{-L}^{L-z}
		\abs{H'\prt{\rho\h(t, x+z)}-H'\prt{\rho\h(t, x)}}^2 \dx\dt
		\leq
		\mathcal{C} \brk{K\Dh + \prt{K+1}\Dl}
		= \mathcal{C} z,
	\end{align}
	using $\Dx=\Dl+\Dh$ and $K\Dx+\Dl=z$. This concludes the proof.
\end{proof}

We now proceed to show compactness.

\begin{proof} [Proof of \cref{thm:compactness_entropic_part}]
	Let $\prt{\rho\h}$ be the piecewise constant interpolations associated to a sequence of solutions to scheme \eqref{eq:discrete}, and consider the family $\prt{H'\prt{\rho\h}}$. The uniform $\Linf$-bound of $\rho\h$ implies that the family $\prt{H'\prt{\rho\h}}$ is also bounded in $\Linf(\QT)$. The uniform bounds in conjunction with the control of the space and time translates from \cref{th:timeshift,th:spaceshift} suffice to invoke the theorem of Kolmogorov-Riesz-Fr\'{e}chet \cite[Theorem 4.26]{Brezis2010}. Thus, the family $\prt{H'\prt{\rho\h}}$ converges strongly in in $\Ltwo\prt{\QT}$ (up to a subsequence) to a function $\chi\in\Ltwo\prt{\QT}$.
\end{proof}

It is easy to see that the compactness of $\rho\h$ follows directly from the strong compactness of $\prt{H'\prt{\rho\h}}_h$.

\begin{theorem} [$\Ltwo\prt{\QT}$-compactness of $\rho\h$] \label{thm:compactness}
	Let $\prt{\rho\h}_h$ be the piecewise constant interpolations associated to a sequence of solutions to scheme \eqref{eq:discrete}. Then, the family converges strongly (up to a subsequence) in the $\Ltwo\prt{\QT}$-sense to a function $\rho\in\Ltwo\prt{\QT}$.
\end{theorem}

\begin{proof}
	The proof is based on a lemma of \cite{H.V.P2000}. For the sake of convenience we provide the reader with a version tailored to our needs, \cref{th:lemmaA} in the appendix.

	\cref{th:aprioriLinfboundrho} yields $\prt{\rho\h}\subset \Linf(\QT)$, uniformly. In particular, this implies the existence of a function, $\rho\in\Ltwo(\QT)$, such that
	\begin{align}
		\rho\h \rightharpoonup \rho, \quad \text{weakly in} \quad \Ltwo(\QT),
	\end{align}
	along a subsequence, using the well-known Banach-Alaoglu theorem. Recalling the previous compactness result, \cref{thm:compactness_entropic_part}, which provides $H'(\rho\h) \to \chi$, we may apply \cref{th:lemmaA} and deduce $\chi = H'(\rho)$. In particular, this means that
	\begin{align}
		H'(\rho\h) \to H'(\rho), \quad \text{strongly in} \quad \Ltwo(\QT),
	\end{align}
	up to a subsequence.

	The $\Ltwo$ convergence implies convergence almost-everywhere of $H'(\rho\h)$ to $H'(\rho)$. At this stage, we use the fact that $H'$ is invertible. Applying the inverse, $\prt{H'}^{-1}$, to $H'(\rho\h) \to H'(\rho)$, we recover almost everywhere convergence of $\rho\h$ to $\rho$. The $\Linf$-bound on $\prt{\rho\h}_h$ permits the application of the dominated convergence theorem, which, in turn, implies
	\begin{align}
		\rho\h \to \rho, \quad \text{strongly in } \Lp(\QT)
	\end{align}{}
	for any $1\leq p< \infty$; in particular, in $\Ltwo(\QT)$.
\end{proof}

Having obtained the convergence of $\rho\h$, we are able to discuss the convergence of the velocity term of the equation.

\begin{lemma} [$\Ltwo\prt{\QT}$-Compactness of Discrete Derivatives] \label{thm:compactnessderivatives}
	Let $\prt{\rho\h}$ be a sequence of solutions to scheme \eqref{eq:discrete} with limit $\rho$ in $\Ltwo\prt{\QT}$, and consider:
	\begin{align}
		\ddx H'(\rho\h)(t, x)         & = \frac{H'(\rho\i\np)-H'(\rho\i\n)}{\Dx},             \\
		\ddx V\h(x)                   & = \frac{V\ip-V\i}{\Dx},                               \\
		\ddx \prt{W\conv\rho}\h(t, x) & = \frac{\prt{W\conv\rho}\ip-\prt{W\conv\rho}\i}{\Dx},
	\end{align}
	for $t\in I\n$ and $x\in C\ih$. Then:
	\begin{enumerate}[(i)]
		\item $H'(\rho)$ belongs to $\Hone\prt{\Omega}$, for all times $t\in\lbrk{0,T}$.
		\item $\ddx H'(\rho\h)$ converges weakly (up to a subsequence) in $\Ltwo\prt{\QT}$ with limit $\pder{}{x} H'(\rho)$.
		\item $\ddx V\h$ (respectively $\ddx\prt{W\conv\rho}\h$) converges strongly (up to a subsequence) in the $\Ltwo\prt{\QT}$-sense to $V'$ (resp. $W'\conv\rho$).
	\end{enumerate}
\end{lemma}

We omit the proof of this lemma as it is identical to the proofs of Lemma 4.2 and Proposition 1 in \cite{C.F.S2020}.
 \section[Convergence of the Scheme] {Convergence of the Scheme: Proof of \cref{th:convergence}} \label{sec:convergence_scheme}

This section is devoted to proving the main result, \cref{th:convergence}. Having established all necessary estimates, we are prepared to prove the convergence of the piecewise constant interpolations, $\prt{\rho_h}$, associated to scheme \eqref{eq:discrete} to weak solutions of \Cref{eq:continuous} in the sense of \Cref{def:weaksolution}.

To begin, let $\varphi$ be a smooth test function such that $\varphi(T)=0$. We define the error term
\begin{align} \label{eq:errorterm}
	\begin{split}
		\epsilon(h)
		&= -\int_{0}^{T} \prt*{ \int_{-L}^{L} \rho\h \pder{\varphi}{t} \dx - \int_{-L+\Dx/2}^{L-\Dx/2} \rho\h \pder{\varphi}{x} \ddx \brk{H'(\rho\h)+V\h+\prt{W\conv\rho}\h} } \dx\dt \\
		&\quad - \int_{-L}^{L} \rho\h(0, x) \varphi(0, x) \dx,
	\end{split}
\end{align}
and observe that, as the mesh size goes to zero, all the terms of this quantity converge, by virtue of \cref{thm:compactness,thm:compactnessderivatives}, to those of the weak solution of \cref{eq:continuous}:
\begin{align}
	\epsilon(h) & \rightarrow
	-\int_{0}^{T} \int_{-L}^{L} \rho \prt*{\pder{\varphi}{t} - \pder{\varphi}{x} \pder{}{x} \brk{H'(\rho)+V+\prt{W\conv\rho}} } \dx\dt
	-\int_{-L}^{L} \rho(0, x) \varphi(0, x) \dx.
\end{align}
In the spirit of \cite{C.F.S2020}, it remains to show that $\epsilon(h) \rightarrow 0$, by comparing Eq. \eqref{eq:errorterm} with the scheme, thereby proving that $\rho$ is, indeed, a weak solution.

In order to prove this claim, we define the following cell averages of the test function, \revisionOne{$\varphi$},
\begin{align}
	\varphi\i(t) \coloneqq \dashint_{C\i} \varphi(t,x) \dx, \quad \varphi\ih(t) \coloneqq \dashint_{C\ih} \varphi(t,x) \dx, \quad \text{as well as} \quad \varphi\i\n = \varphi\i(t\n).
\end{align}
Multiplying scheme \eqref{eq:discrete} by $\varphi\i\np$ and integrating, we obtain
\begin{align} \label{eq:weak_formulation_at_level_of_scheme}
	\begin{split}
		0
		&= \sum_{n=0}^{N} \sum_{i=1}^{2M} \prt*{ \frac{\rho\i\np-\rho\i\n}{\Dt} + \frac{F\ih\np-F\imh\np}{\Dx} } \varphi\i\np \Dx\Dt \\
		&= - \sum_{i=1}^{2M} \prt*{ \sum_{n=1}^{N} \rho\i\n \prt{\varphi\i\np-\varphi\i\n} + \rho\i^{0} \varphi\i^{1}} \Dx
		- \sum_{n=0}^{N} \sum_{i=1}^{2M-1} F\ih\np \ddx \varphi\ih\np \Dx \Dt,
	\end{split}
\end{align}
where we have used summation by parts, $\varphi\i^{N+1} = 0$, and the no-flux boundary conditions, i.e., $F_{1/2}\np = F_{2M + 1/2}\np = 0$. As before, we manipulate the flux term and observe that
\begin{align}
	\lll \sum_{n=0}^{N} \sum_{i=1}^{2M-1} F\ih\np
	\ddx \varphi\ih\np \Dx \Dt                                                                                                                \\
	 & = \sum_{n=0}^{N} \sum_{i=1}^{2M-1} \ddx \varphi\ih\np \prt*{ \rho\i\np\pos{u\ih\np} + \rho\ip\np\neg{u\ih\np} } \Dx \Dt                \\
	 & = \sum_{n=0}^{N} \sum_{i=1}^{2M-1} \ddx \varphi\ih\np \prt*{ \rho\i\np u\ih\np + \prt{\rho\ip\np - \rho\i\np} \neg{u\ih\np} } \Dx \Dt.
\end{align}
Using this, we note that Eq. \eqref{eq:weak_formulation_at_level_of_scheme} can be written as
\begin{align}
	0 = \hat{\mathcal{T}}(h)
	+ \hat{\mathcal{H}}(h)
	+ \hat{\mathcal{V}}(h)
	+ \hat{\mathcal{W}}(h)
	+ \hat{\mathcal{E}}(h),
\end{align}
where
\begin{align}
	\hat{\mathcal{T}}(h) & \coloneqq - \sum_{i=1}^{2M} \prt*{ \sum_{n=1}^{N} \rho\i\n \prt{\varphi\i\np-\varphi\i\n} + \rho\i^{0} \varphi\i^{1}} \Dx, \\
	\hat{\mathcal{H}}(h) & \coloneqq \sum_{n=0}^{N} \sum_{i=1}^{2M-1} \ddx \varphi\ih\np \rho\i\np \ddx H'\prt{\rho\ih\np} \Dx \Dt,                   \\
	\hat{\mathcal{V}}(h) & \coloneqq \sum_{n=0}^{N} \sum_{i=1}^{2M-1} \ddx \varphi\ih\np \rho\i\np \ddx V\ih \Dx \Dt,                                 \\
	\hat{\mathcal{W}}(h) & \coloneqq \sum_{n=0}^{N} \sum_{i=1}^{2M-1} \ddx \varphi\ih\np \rho\i\np \ddx (W\conv\rho^{**})\ih \Dx \Dt,                 \\
	\hat{\mathcal{E}}(h) & \coloneqq - \sum_{n=0}^{N} \sum_{i=1}^{2M-1} \ddx \varphi\ih\np \prt{\rho\ip\np - \rho\i\np} \neg{u\ih\np} \Dx \Dt.
\end{align}
In a similar fashion, we may rewrite Eq. \eqref{eq:errorterm} as
\begin{align}
	\epsilon(h)
	 & = \mathcal{T}(h)
	+ \mathcal{H}(h)
	+ \mathcal{V}(h)
	+ \mathcal{W}(h),
\end{align}
with terms given by
\begin{align}
	\mathcal{T}(h) & \coloneqq -\int_{-L}^{L}\prt*{ \int_{0}^{T} \rho\h\pder{\varphi}{t} \dt + \rho\h(0,x)\varphi(0,x) }\dx,     \\
	\mathcal{H}(h) & \coloneqq \int_{0}^{T} \int_{-L+\Dx/2}^{L-\Dx/2} \rho\h \pder{\varphi}{x} \ddx {H'(\rho\h)} \dx\dt,         \\
	\mathcal{V}(h) & \coloneqq \int_{0}^{T} \int_{-L+\Dx/2}^{L-\Dx/2} \rho\h \pder{\varphi}{x} \ddx {V\h} \dx\dt,                \\
	\mathcal{W}(h) & \coloneqq \int_{0}^{T} \int_{-L+\Dx/2}^{L-\Dx/2} \rho\h \pder{\varphi}{x} \ddx {\prt{W\conv\rho}\h} \dx\dt.
\end{align}
Already noticing the resemblance, we will show that
\begin{align} \label{eq:comparison_scheme_weak_formulation}
	\abs{\mathcal{T}(h)-\hat{\mathcal{T}}(h)}, \,
	\abs{\mathcal{H}(h)-\hat{\mathcal{H}}(h)}, \,
	\abs{\mathcal{V}(h)-\hat{\mathcal{V}}(h)}, \,
	\abs{\mathcal{W}(h)-\hat{\mathcal{W}}(h)}, \,
	\hat{\mathcal{E}}(h)
	\rightarrow 0,
\end{align}
as $h\to 0$.
 \paragraph{The Time Term.}

The first term is exactly zero since
\begin{align}
	\lll \mathcal{T}(h) - \hat{\mathcal{T}}(h)                                                                                                                           \\
	 & = \sum_{i=1}^{2M} \prt*{ \sum_{n=1}^{N} \rho\i\n \prt{\varphi\i\np-\varphi\i\n} + \rho\i^{0} \varphi\i^{1} } \Dx
	-\int_{-L}^{L} \prt*{ \int_{0}^{T} \rho\h \pder{\varphi}{t} \dt + \rho\h(0,x)\varphi(0,x) }\dx                                                                       \\
	 & = \sum_{i=1}^{2M} \Biggl( \sum_{n=1}^{N} \rho\i\n \brk*{ \prt{\varphi\i\np-\varphi\i\n} - \dashint_{C\i} \int_{I\n} \pder{\varphi}{t} \dt\dx }\Biggr)             \\
	 & \quad + \sum_{i=1}^{2M} \biggl(\rho\i^{0} \brk*{ \varphi\i^{1} - \dashint_{C\i} \prt*{ \int_{I^{0}} \pder{\varphi}{t} \dt + \varphi(0,x) } \dx } \Biggr) \Dx = 0.
\end{align}

\paragraph{The Entropic Term.}

We observe
\begin{align}
	\lll\mathcal{H}(h) - \hat{\mathcal{H}}(h)                                                                           \\
	 & = \int_{0}^{T} \int_{-L+\Dx/2}^{L-\Dx/2} \rho\h \pder{\varphi}{x} \ddx {H'(\rho\h)} \dx\dt
	- \sum_{n=0}^{N} \sum_{i=1}^{2M-1} \ddx \varphi\ih\np \rho\i\np \ddx H'\prt{\rho\ih\np} \Dx\Dt                      \\
	 & = \sum_{n=1}^{N} \sum_{i=1}^{2M-1}\prt*{\int_{I\n} \int_{C\ih} \rho\h \pder{\varphi}{x} \ddx {H'(\rho\h)} \dx\dt
		- \ddx \varphi\ih\n \rho\i\n \ddx H'\prt{\rho\ih\n} \Dx\Dt}                                                         \\
	 & \quad + R_{\mathcal{H}}(h),
\end{align}
with remainder given by
\begin{align}
	R_{\mathcal{H}}(h)
	 & \coloneqq \sum_{i=1}^{2M-1}\int_{I^{0}} \int_{C\ih} \rho\h \pder{\varphi}{x} \ddx {H'(\rho\h)} \dx\dt \\
	 & \quad - \sum_{i=1}^{2M-1} \ddx \varphi\ih^{N+1} \rho\i^{N+1} \ddx H'\prt{\rho\ih^{N+1}} \Dx\Dt.
\end{align}
This can be readily controlled as
\begin{align}
	\abs{R_{\mathcal{H}}(h)}
	 & \leq \int_{I^{0}} \int_{-L+\Dx/2}^{L-\Dx/2} \abs*{\rho\h \pder{\varphi}{x} \ddx {H'(\rho\h)} }\dx\dt       \\
	 & \quad + \sum_{i=1}^{2M-1} \abs*{ \ddx\varphi\ih^{N+1} \rho\i^{N+1} \ddx H'\prt{\rho\ih^{N+1}}} \Dx\Dt      \\
	 & \leq 2\sqrt{2L} \curlyC_\infty \psnorm*{\infty}{\QT}{\pder{\varphi}{x}} \curlyC_{\partial_x H'}^{1/2} \Dt,
\end{align}
by using the Cauchy-Schwarz inequality on both terms, as well as the fact that
\begin{align}
	\abs*{\ddx\varphi\ih\n}
	\leq \psnorm*{\infty}{\Omega}{\pder{\varphi}{x}}.
\end{align}

Having dealt with the remainder, we find
\begin{align}
	\lll \mathcal{H}(h) - \hat{\mathcal{H}}(h) - R_{\mathcal{H}}(h)                                                                                                                        \\
	 & = \sum_{n=1}^{N} \sum_{i=1}^{2M-1}\biggl(\int_{I\n} \int_{C\ih} \rho\h \pder{\varphi}{x} \ddx {H'(\rho\h)} \dx\dt - \ddx \varphi\ih\n \rho\i\n \ddx H'\prt{\rho\ih\n} \Dx\Dt\biggr) \\
	 & = \sum_{n=1}^{N} \sum_{i=1}^{2M-1}
	\ddx H'\prt{\rho\ih\n} \Biggl(\int_{I\n} \brk*{\rho\i\n \int_{C\i\cap C\ih} \pder{\varphi}{x} \dx + \rho\ip\n \int_{C\ih\cap C\ip} \pder{\varphi}{x} \dx} \dt                          \\
	 & \quad - \ddx \varphi\ih\n \rho\i\n \Dx\Dt \vphantom{\int_{I\n} \brk*{\rho\i\n \int_{C\i\cap C\ih} \pder{\varphi}{x} \dx}} \Biggr)                                                   \\
	 & = \sum_{n=1}^{N} \sum_{i=1}^{2M-1} \ddx H'\prt{\rho\ih\n} \int_{I\n} \Bigl( \rho\i\n \brk*{\varphi\prt{t,x\ih} - \varphi\prt{t,x\i}}                                                \\
	 & \quad + \rho\ip\n \brk*{\varphi\prt{t,x\ip} - \varphi\prt{t,x\ih}} - \rho\i\n\brk{\varphi\ip\n - \varphi\i\n} \Bigr) \dt                                                            \\
	 & = \sum_{n=1}^{N} \sum_{i=1}^{2M-1} \ddx H'\prt{\rho\ih\n} \rho\i\n \int_{I\n} \brk*{\varphi\prt{t,x\ip} - \varphi\prt{t,x\i}} - \brk*{\varphi\ip\n - \varphi\i\n} \dt               \\
	 & \quad + \sum_{n=1}^{N} \sum_{i=1}^{2M-1} \ddx H'\prt{\rho\ih\n} \prt{\rho\ip\n-\rho\i\n} \int_{I\n} \brk*{\varphi\prt{t,x\ip} - \varphi\prt{t,x\ih}} \dt.
\end{align}

The first term in the sum can be shown to vanish in the limit by noting that
\begin{align}
	\lll \abs*{\brk*{\varphi\prt{t,x\ip} - \varphi\prt{t,x\i}} - \brk*{\varphi\ip\n - \varphi\i\n}}                                                \\
	 & = \abs*{\frac{\varphi\prt{t,x\ip} - \varphi\prt{t,x\i}}{\Dx} - \dashint_{C\i}\frac{\varphi\prt{t\n,\Dx+s} - \varphi\prt{t\n,s}}{\Dx}\ds}\Dx \\
	 & = \abs*{\pder{\varphi}{x}\prt{t,\theta\ih} - \frac{\varphi\prt{t\n,\Dx+\hat\theta\i} - \varphi\prt{t\n,\hat\theta\i}}{\Dx}\ds}\Dx           \\
	 & = \abs*{\pder{\varphi}{x}\prt{t,\theta\ih} - \pder{\varphi}{x}\prt{t\n,\tilde\theta\ih}}\Dx                                                 \\
	 & \leq \prt*{ \psnorm*{\infty}{\QT}{\secondpder{\varphi}{x}} + \psnorm*{\infty}{\QT}{\doublepder{\varphi}{t}{x}}} \prt*{\Dt + \Dx} \Dx,
\end{align}
for some constants $\theta\ih\in C\ih$, $\hat\theta\i\in C\i$, $\tilde\theta\ih\in C\i\cap C\ip$. Thus we control the term by
\begin{align}
	\lll \sum_{n=1}^{N} \sum_{i=1}^{2M-1} \ddx H'\prt{\rho\ih\n} \rho\i\n \int_{I\n} \brk*{\varphi\prt{t,x\ip} - \varphi\prt{t,x\i}} - \brk*{\varphi\ip\n - \varphi\i\n} \dt \\
	 & \leq \mathcal{C}\prt{\Dt + \Dx}\sum_{n=1}^{N} \sum_{i=1}^{2M-1} \ddx H'\prt{\rho\ih\n} \rho\i\n \Dx\Dt                                                                \\
	 & \leq \mathcal{C}\sqrt{2LT}\prt{\Dt + \Dx}\curlyC_\infty \curlyC_{\partial_x H'}^{1/2},
\end{align}
which vanishes as the mesh size goes to zero.

Handling the second sum in a similar fashion, we see
\begin{align}
	\lll \sum_{n=1}^{N} \sum_{i=1}^{2M-1} \ddx H'\prt{\rho\ih\n} \prt{\rho\ip\n-\rho\i\n} \int_{I\n} \brk*{\varphi\prt{t,x\ip} - \varphi\prt{t,x\ih}} \dt                                 \\
	 & \leq \psnorm*{\infty}{\QT}{\pder{\varphi}{x}} \sum_{n=1}^{N} \sum_{i=1}^{2M-1} \ddx H'\prt{\rho\ih\n} \prt{\rho\ip\n-\rho\i\n}\Dx\Dt                                               \\
	 & \leq \psnorm*{\infty}{\QT}{\pder{\varphi}{x}} \psnorm{2}{\QT}{\ddx H'\prt{\rho\h}} \prt*{ \int_{0}^{T} \int_{L}^{L-\Dx} \abs*{ \rho\h(t, x+\Dx) - \rho\h(t, x) }^2 \dx \dt}^{1/2},
\end{align}
through the use of Cauchy-Schwarz. As discussed previously, this integral term tends to zero with the mesh size, due to the convergence of $\rho$. Hence, we conclude $\abs{\mathcal{H}(h)-\hat{\mathcal{H}}(h)}\rightarrow 0$, as claimed.

\paragraph{The Potential Terms.}

The prior computation applied to the external potential term yields
\begin{align}
	\lll \mathcal{V}(h)-\hat{\mathcal{V}}(h)                                                                                                                                        \\
	 & = \int_{0}^{T} \int_{-L+\Dx/2}^{L-\Dx/2} \rho\h \pder{\varphi}{x} \ddx {V\h} \dx\dt - \sum_{n=0}^{N} \sum_{i=1}^{2M-1} \ddx \varphi\ih\np \rho\i\np \ddx V\ih \Dx \Dt        \\
	 & = \sum_{n=1}^{N} \sum_{i=1}^{2M-1}\prt*{\int_{I\n} \int_{C\ih} \rho\h \pder{\varphi}{x} \ddx V\h \dx\dt - \ddx \varphi\ih\n \rho\i\n \ddx V\ih \Dx\Dt} + R_{\mathcal{V}}(h),
\end{align}
with remainder term
\begin{align}
	R_{\mathcal{V}}(h) \coloneqq
	\sum_{i=1}^{2M-1}\int_{I^{0}} \int_{C\ih} \rho\h \pder{\varphi}{x} \ddx {V\h} \dx\dt
	- \sum_{i=1}^{2M-1} \ddx \varphi\ih^{N+1} \rho\i^{N+1} \ddx V\ih \Dx\Dt,
\end{align}
which is controlled by
\begin{align}
	\abs{R_{\mathcal{V}}(h)}
	 & \leq 4L \curlyC_\infty \psnorm*{\infty}{\QT}{\pder{\varphi}{x}} \psnorm*{\infty}{\Omega}{V'} \Dt,
\end{align}
the analogue of the bound for the previous term. This is shown by employing the $\Linf$-bound on $V'$ in place of the $\Ltwo$ bound on the entropic term. Following the same strategy, we find
\begin{align}
	\mathcal{V}(h) - \hat{\mathcal{V}}(h) - R_{\mathcal{V}}(h)
	 & = \sum_{n=1}^{N} \sum_{i=1}^{2M-1} \ddx V\ih \rho\i\n \int_{I\n} \brk*{\varphi\prt{t,x\ip} - \varphi\prt{t,x\i}} - \brk*{\varphi\ip\n - \varphi\i\n} \dt \\
	 & \quad + \sum_{n=1}^{N} \sum_{i=1}^{2M-1} \ddx V\ih \prt{\rho\ip\n-\rho\i\n} \int_{I\n} \brk*{\varphi\prt{t,x\ip} - \varphi\prt{t,x\ih}} \dt.
\end{align}
Each of the sums is shown to tend to zero as the mesh size goes to zero, just as previously. Thus we see $\abs{\mathcal{V}(h)-\hat{\mathcal{V}}(h)}\rightarrow 0$.

The discussion of the interaction term is identical:
\begin{align}
	\mathcal{W}(h) - \hat{\mathcal{W}}(h)
	 & = \sum_{n=1}^{N} \sum_{i=1}^{2M-1} \ddx (W\conv\rho^{**})\ih \rho\i\n \int_{I\n} \brk*{\varphi\prt{t,x\ip} - \varphi\prt{t,x\i}} - \brk*{\varphi\ip\n - \varphi\i\n} \dt \\
	 & \quad + \sum_{n=1}^{N} \sum_{i=1}^{2M-1} \ddx (W\conv\rho^{**})\ih \int_{I\n} \brk*{\varphi\prt{t,x\ip} - \varphi\prt{t,x\ih}} \dt + R_{\mathcal{W}}(h).
\end{align}
Each of the sums (as well as the remainder $R_{\mathcal{W}}(h)$) are shown to vanish in the limit, proving $\abs{\mathcal{W}(h)-\hat{\mathcal{W}}(h)}\rightarrow 0$.

\paragraph{The Error Term.}

Finally, we show a control on $\hat{\mathcal{E}}(h)$ using some of the previous estimates:
\begin{align}
	\hat{\mathcal{E}}(h)
	 & \leq \sum_{n=0}^{N} \sum_{i=1}^{2M-1} \abs{ \ddx \varphi\ih\np } \abs{ \rho\ip\np - \rho\i\np } \abs{ u\ih\np } \Dx \Dt                                               \\
	 & \leq \psnorm*{\infty}{\OmegaT}{\pder{\varphi}{x}} \sum_{n=0}^{N} \sum_{i=1}^{2M-1} \abs{ \rho\ip\np - \rho\i\np } \abs{ u\ih\np } \Dx \Dt                             \\
	 & \leq \psnorm*{\infty}{\OmegaT}{\pder{\varphi}{x}} \sum_{n=0}^{N} \sum_{i=1}^{2M-1} \abs{ \rho\ip\np - \rho\i\np } \prt*{ \abs{ h\ih\np } + \abs{ v\ih\np } } \Dx \Dt,
\end{align}
splitting once again the velocity term
as was done in the proof of \cref{th:aprioriLinfboundrho},
\begin{align}
	\hat{\mathcal{E}}(h)
	 & \leq
	\psnorm*{\infty}{\OmegaT}{\pder{\varphi}{x}} \Biggl[
		\prt*{\sum_{n=0}^{N} \sum_{i=1}^{2M-1} \abs{ \rho\ip\np - \rho\i\np }^2 \Dx \Dt}^{1/2}
		\prt*{\sum_{n=0}^{N} \sum_{i=1}^{2M-1} \abs{ h\ih\np }^2 \Dx \Dt}^{1/2}                                                                                                               \\
	 & \quad + \prt*{\psnorm{\infty}{\Omega}{V'} + \psnorm{\infty}{\Omega}{W'} \psnorm{1}{\Omega}{\rho_0}} \sum_{n=0}^{N} \sum_{i=1}^{2M-1} \abs{ \rho\ip\np - \rho\i\np } \Dx \Dt\Biggr] \\
	 & \leq \mathcal{C} \prt*{\sum_{n=0}^{N} \sum_{i=1}^{2M-1} \abs{ \rho\ip\np - \rho\i\np }^2 \Dx \Dt}^{1/2},
\end{align}
by the repeated application of the Cauchy-Schwarz inequality, where
\begin{align}
	\psnorm*{\infty}{\OmegaT}{\pder{\varphi}{x}}
	\prt*{
		\curlyC_{\partial_x H'}^{1/2}
		+ \sqrt{LT} \prt*{ \psnorm{\infty}{\Omega}{V'} + \psnorm{\infty}{\Omega}{W'} \psnorm{1}{\Omega}{\rho_0} }
	}\leq \mathcal{C},
\end{align}
a constant independent of the mesh (see \cref{th:potentialderivative,th:aprioriL2boundgradH}). We thus find
\begin{align}
	\hat{\mathcal{E}}(h)
	 & \leq \mathcal{C} \prt*{\sum_{n=0}^{N} \sum_{i=1}^{2M-1} \abs{ \rho\ip\np - \rho\i\np }^2 \Dx \Dt}^{1/2} \\
	 & = \mathcal{C} \prt*{
		\int_{0}^{T}
		\int_{L}^{L-\Dx}
		\abs*{ \rho\h(t+\Dt, x+\Dx) - \rho\h(t+\Dt, x) }^2 \dx \dt}^{1/2},
\end{align}
a term which vanishes with the mesh size $h$ due to the convergence of $\rho\h$.

\paragraph{Convergence.}

Combining the results of the preceding steps proves that \eqref{eq:comparison_scheme_weak_formulation} holds indeed. Therefore, we conclude that $\epsilon(h)\rightarrow 0$, which proves the convergence of the scheme with limit a weak solution to \cref{eq:continuous} in the sense of \cref{def:weaksolution}, as claimed in \cref{th:convergence}.
 \section{The Case of Linear Diffusion} \label{sec:linear}

The discussion of the convergence of scheme \eqref{eq:discrete} cannot be complete without addressing the case of linear diffusion. The choice of Boltzmann's entropy, $H(\rho) = \rho \log\prt{\rho} - \rho$, renders \cref{eq:continuous} into the linear drift-diffusion equation:
\begin{align} \label{eq:continuous_linear}
	\pder{\rho}{t} = \secondpder{\rho}{x} + \pder{}{x}\brk*{\rho\pder{}{x}\prt{V(x)+W(x)\conv\rho}}.
\end{align}
\Cref{eq:continuous_linear} is within the purview of \cite{McCann1997,Otto2001,C.M.V2003,C.M.V2006}, and the variational structure \eqref{eq:energydissipation} persists.

However, the choice of entropy $H(\rho) = \rho \log\prt{\rho} - \rho$ cannot be addressed by the main result, \cref{th:convergence}. The derivative $H'(\rho) = \log\prt{\rho}$ lacks Lipschitz continuity at zero, which causes the estimates of \cref{th:timeshift,th:spaceshift} to break down. In fact, the problem is even more blatant: no matter the approach, one must control the quantity $\rho \px \log\prt{\rho}$ in order to reproduce the estimates of \cref{sec:convergence_scheme}. At the continuous level, this is equivalent to controlling $\px \rho$; however, in the discrete case, the quantity appears as
\begin{align} \label{eq:trouble_quantity}
	\rho\i\np\prt*{\frac{\log\prt{\rho\ip\np}-\log\prt{\rho\i\np}}{\Dx}},
\end{align}
which may be unbounded.

The only viable approach is to apply the scheme to positive initial data, $\rho_0(x) \geq \varepsilon$ \revisionTwo{for all $x\in\Omega$}, for some $\varepsilon >0$. Using the propagation of lower bounds from \cref{th:aprioriLinfboundrho}, we will infer the boundedness of the solution away from zero for all times, $\rho\i\n \geq\kappa^{-1}>0$, which permits the control of terms like \eqref{eq:trouble_quantity} through
\begin{align} \label{eq:trouble_quantity_control}
	\rho\i\np\abs*{\frac{\log\prt{\rho\ip\np}-\log\prt{\rho\i\np}}{\Dx}}
	\leq \curlyC_\infty \kappa \abs*{\frac{\rho\ip\np-\rho\i\np}{\Dx}}.
\end{align}

Rather than obtaining the compactness of the family $H\prt{\rho\h}$ through the auxiliary functional $K$, we will obtain the compactness of $\prt{\rho_h}$ directly through the classical functional $\pnormp{2}{\rho}$. To be precise, we find
\begin{align}
	\sum_{i=1}^{2M} \prt*{ \prt*{\rho\i\np}^2 - \prt*{\rho\i\n}^2 }
	 & \leq \sum_{i=1}^{2M} \rho\i\np \prt*{\rho\i\np - \rho\i\n}
	=- \Dt \sum_{i=1}^{2M} \ddx \rho\ih\np F\ih\np                                   \\
	 & \leq \alpha^{-1} C - (1-\alpha) \Dt \sum_{i=1}^{2M} \abs*{\ddx \rho\ih\np}^2,
\end{align}
where the last line is obtained from the bounds on the potentials $V, W$, and the application of Young's inequality with parameter $\alpha\in(0,1)$. This readily yields a bound
\begin{align} \label{eq:linearL2}
	\psnormp{2}{\OmegaT}{\ddx \rho\h}\leq \cdR,
\end{align}
in the style of \cref{th:aprioriL2boundgradH}.

Armed with gradient information, we find bounds for the shifts in time and space of $\rho\h$ which are equivalent to those of \cref{th:timeshift,th:spaceshift}. The proof of the space shift is identical, and the time shift argument is parallel until the term \eqref{eq:trouble_quantity} appears, which is readily controlled as in \eqref{eq:trouble_quantity_control}. For completeness, we sketch the idea while neglecting the potential terms, which can be incorporated using Young's inequality. After introducing the discretisation from \cref{eq:cmptimeshift}, we find
\begin{align}
	\lll
	\sum_{n=0}^{N-K}
	\sum_{i=1}^{2M}
	\abs{\rho\i\K-\rho\i\n}^2
	\Dx\Dh    \\
	 & =
	\sum_{n=0}^{N-K}
	\sum_{i=1}^{2M}
	\prt{\rho\i\K-\rho\i\n}
	\sum_{k=1}^{K}
	\frac{\rho\i\nk-\rho\i\nkm}{\Dt}
	\Dt\Dx\Dh \\
	 & =
	-
	\sum_{n=0}^{N-K}
	\sum_{i=1}^{2M}
	\prt{\rho\i\K-\rho\i\n}
	\sum_{k=1}^{K}
	\ddx F\i\nk
	\Dt\Dx\Dh
	\\
	 & =
	\sum_{n=0}^{N-K}
	\sum_{i=1}^{2M-1}
	\prt{\ddx\rho\ih\K-\ddx\rho\ih\n}
	\sum_{k=1}^{K}
	F\ih\nk
	\Dt\Dx\Dh \\
	 & =
	\sum_{n=0}^{N-K}
	\sum_{i=1}^{2M-1}
	\sum_{k=1}^{K}
	\prt{\ddx\rho\ih\K-\ddx\rho\ih\n}
	\prt{\rho\i\nk\pos{h\ih\nk} + \rho\ip\nk\neg{h\ih\nk}}
	\Dt\Dx\Dh
	.
\end{align}
Using the control from \eqref{eq:trouble_quantity_control}, we observe
\begin{align}
	\abs*{\rho\i\nk\pos{h\ih\nk} + \rho\ip\nk\neg{h\ih\nk}}
	\leq 2 \curlyC_\infty \kappa \abs{\ddx\rho\ih\nk},
\end{align}
which simply yields
\begin{align}
	\lll
	\sum_{n=0}^{N-K}
	\sum_{i=1}^{2M}
	\abs{\rho\i\K-\rho\i\n}^2
	\Dx\Dh                                          \\
	 & =
	\sum_{n=0}^{N-K}
	\sum_{i=1}^{2M-1}
	\sum_{k=1}^{K}
	\prt{\ddx\rho\ih\K-\ddx\rho\ih\n}
	\prt{\rho\i\nk\pos{h\ih\nk} + \rho\ip\nk\neg{h\ih\nk}}
	\Dt\Dx\Dh                                       \\
	 & \leq 2\curlyC_\infty \kappa \sum_{n=0}^{N-K}
	\sum_{i=1}^{2M-1}
	\sum_{k=1}^{K}
	\prt{\abs{\ddx\rho\ih\K}+\abs{\ddx\rho\ih\n}}
	\abs{\ddx\rho\ih\nk}
	\Dt\Dx\Dh.
\end{align}
Expanding the product term by term, using Young's inequality, and extending the sum over $n$, we finally obtain
\begin{align}
	\sum_{n=0}^{N-K}
	\sum_{i=1}^{2M}
	\abs{\rho\i\K-\rho\i\n}^2
	\Dx\Dh
	\leq 2 \curlyC_\infty \kappa \sum_{n=0}^{N}
	\sum_{i=1}^{2M-1}
	\sum_{k=1}^{K}
	\abs{\ddx\rho\ih\n}^2
	\Dt\Dx\Dh
	\leq \curlyC K\Dh,
\end{align}
which is analogous to the bound found in \cref{th:timeshift}, and permits the completion of the proof in the same fashion.

The compactness of $\prt{\rho\h}$ on $\QT$ now follows from the Frechet-Kolmogorov-Riesz theorem, and the weak compactness of the derivatives, $\prt{\ddx\rho\h}$, once again follows from Proposition 1 in \cite{C.F.S2020}. This is enough to pass to the limit by reproducing the estimates of \cref{sec:convergence_scheme}, using the newly found estimates \eqref{eq:trouble_quantity_control} and \eqref{eq:linearL2}, proving that \eqref{eq:comparison_scheme_weak_formulation} holds indeed. Therefore, we conclude the convergence of the scheme:

\begin{customthm}{\ref{th:convergence} revisited}[Convergence of the scheme]
	Suppose that $H(\rho) = \rho \log\prt{\rho} - \rho$ and that $V, W \in C^2([-L, L])$. Then, for any strictly positive initial datum $\rho_0\in\Linf(\Omega)$, we find that:
	\begin{enumerate}[(i)]
		\item scheme \eqref{eq:discrete} admits a strictly positive solution that preserves the initial mass regardless of the mesh size;
		\item under the condition ${\curlyC}_V^{(2)}\Dt < 1$ (viz. \cref{th:aprioriLinfboundrho}), the associated piecewise constant interpolation, $\rho_h$, converges strongly in any $\Lp(\QT)$, for $1\leq p<\infty$, up to a subsequence;
		\item the limit is a weak solution to \Cref{eq:continuous} in the sense of \Cref{def:weaksolution}.
	\end{enumerate}
\end{customthm}

\begin{remark}[Non-strictly positive data]
	If the initial datum $\rho_0$ is positive but not strictly so, we may approximate it by $\rho_0^{\varepsilon}(x) \coloneqq \rho_0(x) + \varepsilon$, for some $\varepsilon>0$. Then, choosing $\varepsilon$ and the mesh size $h$ sufficiently small, and using the continuity of the aggregation-diffusion equation \eqref{eq:continuous_linear} with respect to the datum, we nevertheless obtain the convergence of the numerical scheme.
\end{remark}
 
\section*{Acknowledgements}

JAC was partially supported by EPSRC grant  EP/P031587/1 and the Advanced Grant Nonlocal-CPD (Nonlocal PDEs for Complex Particle Dynamics:
Phase Transitions, Patterns and Synchronization) of the European Research Council Executive Agency (ERC) under the European Union's Horizon 2020 research and innovation programme (grant agreement No. 883363).
HM was partially supported by JSPS KAKENHI grant numbers 17K05368 and 18H01139, and JST CREST grant number JPMJCR14D3. MS fondly acknowledges the support of the Fondation Sciences Math\'ematiques de Paris (FSMP) for the postdoctoral fellowship. \singleappendix

For the sake of a complete exposition we reproduce a lemma of \cite{H.V.P2000} in a form tailored to our needs.
\begin{lemma}
	\label{th:lemmaA}
	{\normalfont \cite[Lemma A]{H.V.P2000}}
	Let $\prt{u_n} \subset \Linf(\QT)$ and $f \in C(\R)$ be such that:
	\begin{itemize}
		\item $u_n \rightharpoonup u$, weakly in $\Ltwo(\QT)$;
		\item $f$ is non-decreasing;
		\item $f(u_n) \rightarrow \chi$, strongly in $L^2(\QT)$.
	\end{itemize}{}
	Then $\chi = f(u)$.
\end{lemma}

Also for completeness, we recall the technique of discrete integration by parts.
\begin{lemma}[Summation by Parts]\label{th:summationbyparts}
	Let $\prt{a_i}$ and $\prt{b_i}$ be sequences. Then, for $m, n \in \mathbb{N}$, it holds
	\begin{align}
		\sum_{i=m}^{n} a\i \prt{b\ip-b\i} =
		-\sum_{i=m}^{n-1} \prt{a\ip-a\i} b\ip
		+ a_{n}b_{n+1} - a_{m}b_{m}.
	\end{align}
\end{lemma} \bibliographystyle{abbrv}
\bibliography{./BailoCarrilloMurakawaSchmidtchen_ConvergenceEnergyDissipatingScheme.bib}
 
\end{document}